\documentclass[a4paper,12pt]{article}
\usepackage{amsmath}
\usepackage{amsthm}
\usepackage{amssymb}
\usepackage{amscd}
\usepackage{graphicx}
\usepackage{epsfig}
\usepackage[matrix,arrow,curve]{xy}
\usepackage{color}
\usepackage{mathrsfs}
\usepackage[scr=rsfs,cal=boondox]{mathalpha}

\usepackage{bbm}
\usepackage{stmaryrd}
\usepackage{hyperref}
\usepackage{multirow}

\usepackage{latexsym}
\usepackage{amsfonts}
\input xy
\usepackage{tikz}
\usetikzlibrary{matrix}

\newtheorem{theorem}{Theorem}[section]
\newtheorem{proposition}[theorem]{Proposition}
\newtheorem{lemma}[theorem]{Lemma}
\newtheorem{corollary}[theorem]{Corollary}
\newtheorem{conjecture}[theorem]{Conjecture}

\theoremstyle{definition}
\newtheorem{example}[theorem]{Example}
\newtheorem{remark}[theorem]{Remark}
\newtheorem{definition}[theorem]{Definition}

\mathchardef\za="710B  
\mathchardef\zb="710C  
\mathchardef\zg="710D  
\mathchardef\zd="710E  
\mathchardef\zve="710F 
\mathchardef\zz="7110  
\mathchardef\zh="7111  
\mathchardef\zvy="7112 
\mathchardef\zi="7113  
\mathchardef\zk="7114  
\mathchardef\zl="7115  
\mathchardef\zm="7116  
\mathchardef\zn="7117  
\mathchardef\zx="7118  
\mathchardef\zp="7119  
\mathchardef\zr="711A  
\mathchardef\zs="711B  
\mathchardef\zt="711C  
\mathchardef\zu="711D  
\mathchardef\zvf="711E 
\mathchardef\zq="711F  
\mathchardef\zc="7120  
\mathchardef\zw="7121  
\mathchardef\ze="7122  
\mathchardef\zy="7123  
\mathchardef\zf="7124  
\mathchardef\zvr="7125 
\mathchardef\zvs="7126 
\mathchardef\zf="7127  
\mathchardef\zG="7000  
\mathchardef\zD="7001  
\mathchardef\zY="7002  
\mathchardef\zL="7003  
\mathchardef\zX="7004  
\mathchardef\zP="7005  
\mathchardef\zS="7006  
\mathchardef\zU="7007  
\mathchardef\zF="7008  
\mathchardef\zW="700A  

\newcommand{\be}{\begin{equation}}
\newcommand{\ee}{\end{equation}}

\newcommand{\bea}{\begin{eqnarray}}
\newcommand{\eea}{\end{eqnarray}}
\newcommand{\beas}{\begin{eqnarray*}}
\newcommand{\eeas}{\end{eqnarray*}}
\def\*{{\textstyle *}}
\newcommand{\T}{{\mathbb T}}

\newcommand{\SU}{\mathsf{SU}}
\newcommand{\su}{\mathfrak{su}}

\newcommand{\we}{\wedge}
\newcommand{\nn}{\nonumber}
\newcommand{\ot}{\otimes}

\newcommand{\g}{\mathfrak{g}}
\newcommand{\h}{\mathfrak{h}}

\newcommand{\La}{\big\langle}
\newcommand{\Ra}{\big\rangle}

\newcommand{\N}{\mathbb{N}}
\newcommand{\Z}{\mathbb{Z}}
\newcommand{\R}{\mathbb{R}}

\newcommand{\C}{\mathbb{C}}

\newcommand{\pa}{\partial}
\newcommand{\ti}{\times}

\newcommand{\cG}{{\mathcal G}}

\newcommand{\ad}{{\mathrm{ad}}}

\newcommand{\Ll}{{\pounds}}

\def\ran{\rangle}

\def\op{\oplus}

\def\cR{{\mathcal R}}

\def\cH{{\mathcal H}}
\def\cO{{\mathcal O}}

\def\ul{\underline}

\def\Ad{\operatorname{Ad}}

\def\si{\mathsf{i}}

\def\la{\langle}

\def\sD{{\mathsf D}}

\def\sH{{\mathsf H}}

\def\sT{{\mathsf T}}

\def\sv{{\mathsf v}}

\def\xd{\mathrm{d}}
\def\xi{\tx{i}}

\def\cF{{\mathcal F}}

\newdir{|>}{%
!/4.5pt/@{|}*:(1,-.2)@^{>}*:(1,+.2)@_{>}}

\def\dim{\operatorname{dim}}

\def\Ann{\mathsf{Ann}}

\def\gl{\operatorname{gl}}

\def\L{\mathbb{L}}
\def\hL{{\hat{\,\mathbb{L}}}}
\def\U{{\mathsf{U}}}

\def\u{{\mathfrak{u}}}
\def\P{\mathbf{P}}

\newdir{ (}{{}*!/-5pt/@^{(}}

\newcommand{\tr}{\mbox{$\mathrm{tr}$}}

\newcommand{\bk}[2]{\ensuremath{\langle #1 , #2\rangle}}

\newcommand{\Bk}[2]{\ensuremath{\La #1 , #2\Ra}}

\newcommand{\bkk}[2]{\ensuremath{\langle #1 \,|\, #2\rangle}}

\newcommand{\kb}[2]{\ensuremath{| #1\rangle\!\langle #2 |}}

\newcommand{\im}{{\mathsf{Im}}}
\newcommand{\re}{{\mathsf{Re}}}

\newcommand{\vol}{\textnormal{vol}}

\def\N{\mathbb{N}}

\def\t{\mathsf{t}}

\def\Lie{\mathrm{Lie}}
\def\n{\nabla}
\def\g{\mathfrak{g}}
\newcommand{\m}{{\medskip}}
\newcommand{\mn}{{\medskip\noindent}}

\newcommand{\no}{{\noindent}}

\newcommand{\us}{{\u^*}}
\newcommand{\half}{{\frac{1}{2}}}

\def\sv{\mathsf{v}}

\begin{document}
\title{\textbf{\Large Contactifications:
a Lagrangian description\\ of compact Hamiltonian systems}\footnote{The research of K.~Grabowska and J.~Grabowski was partially funded by the National Science Centre (Poland) within the project WEAVE-UNISONO, No. 2023/05/Y/ST1/00043.}}
\author{Katarzyna Grabowska\footnote{email:konieczn@fuw.edu.pl }\\
\textit{Faculty of Physics,
                University of Warsaw}
\\ \\
Janusz Grabowski\footnote{email: jagrab@impan.pl} \\
\textit{Institute of Mathematics, Polish Academy of Sciences}
\\ \\
Marek Ku\'s\footnote{email: marek.kus@cft.edu.pl}\\
\textit{Center for Theoretical Physics, Polish Academy of Sciences}
 \\ \\
Giuseppe Marmo\footnote{email: marmo@na.infn.it}\\
\textit{Dipartimento di Fisica ``Ettore Pancini'', Universit\`{a} ``Federico II'' di Napoli} \\
\textit{and Istituto Nazionale di Fisica Nucleare, Sezione di Napoli} }
\date{}
\maketitle
\begin{abstract}
If $\zh$ is a contact form on a manifold $M$ such that the orbits of the Reeb vector field $\cR$ form a simple foliation  $\cF$ on $M$, then the presymplectic 2-form $\xd\zh$ on $M$ induces a symplectic structure $\zw$ on the quotient manifold $N=M/\cF$. We call $(M,\zh)$ a \emph{contactification} of the symplectic manifold $(N,\zw)$. First, we present an explicit geometric construction of contactifications of some coadjoint orbits of connected Lie groups. Our construction is a far going generalization of the well-known contactification of the complex projective space $\C\P^{n-1}$, being the unit sphere $S^{2n-1}$ in $\C^{n}$, and equipped with the restriction of the Liouville 1-form on $\C^n$. Second, we describe a constructive procedure of obtaining contactification in the process of the Marsden-Weinstein-Meyer symplectic reduction, and indicate geometric obstructions for the existence of compact contactifications. Third,
we show that contactifications provide a nice geometrical tool for a Lagrangian description of Hamiltonian systems on compact symplectic manifolds $(N,\zw)$, on which symplectic forms never admit a `vector potential'.

\par


\vskip1cm
\
\end{abstract}
\section{Introduction}
Nowadays, it is widely believed that the description of fundamental interactions has to be formulated in terms of gauge theories. The most paradigmatic example is provided by electrodynamics. In general, the joint evolution equations of fields and sources gives rise to equations which are essentially untractable when it comes to find explicit solutions. Many simplifying assumptions are usually made to uncouple the evolution of fields and the evolution of particles.

\medskip To this aim, one introduces the notion of test particles, i.e., particles which move under the influence of external fields while their motion does not contribute to the field itself. The evolution of the inner degrees of freedom  of test particles (say, spin, isospin, colour) is usually  described by a Hamiltonian vector field on a compact symplectic manifold. This means: \emph{via} a nondegenerate closed two form which is closed but not exact,
however, the interaction with external fields, usually minimal coupling, requires the existence of a `potential' for the symplectic structure. 

\medskip Actually, there are many finite-dimensional interesting physical systems, which Hamiltonian evolution is described by means of  symplectic structures which are closed but not exact (charged particles in the field of magnetic monopoles, spinning particles, particles whose inner structure is described by orbits of the coadjoint action of compact groups. Other instances of this problem arise in the description of the Berry phase for mixed states, when the eigenvalues of the density operator are pairwise  rationally related among them.

\medskip For such systems it is not possible to provide a global Lagrangian description, as the interaction with external fields requires the existence of a `vector potential' for the symplectic structure. In several specific cases, this has been achieved by replacing the compact symplectic manifold with a compact covering, on which the pull-back of the closed two-form turns out to be exact. For some cases, the problem was addressed by Balachandran and collaborators \cite{Balachandran:1983,Balachandran:1991, Balachandran:2017} and afterwards, in the mathematics community, e.g., by Sternberg and Weinstein \cite{Sternberg:1977,Weinstein:1978}. This is closely related to problems of geometric quantization (see, e.g., \cite{Bates:1997}), in which the covering appears to be an $S^1$- (equivalently, a Hermitian complex line bundle) or $\R$-principal bundle \cite{Boothby:1958,Grabowska:2023}, on which the lifted 2-form is invariant. Note, however, that there is a mistake in the proof in \cite{Boothby:1958} that requires its substantial modification (cf. \cite{Geiges:2008,Grabowska:2023}).

\medskip In this paper, we would like to reconsider the problem of `unfolding', i.e., how to construct a compact `covering' of a compact symplectic manifold on which the pull-back of the two form will be exact, and to do this in the most economical way, i.e., adding only one additional dimension. From the geometrical point of view, this amounts to a contactification of a symplectic compact manifold which is a coadjoint orbit of a compact Lie group.
Since the unfolded 2-form is a presymplectic form being the differential of a contact 1-form, we call this procedure a \emph{contactification} of a symplectic manifold, which goes in the direction opposite to \emph{symplectization} (cf. \cite{Arnold:1989}).
Moreover, we describe a constructive way of obtaining contactifications `on the fly', when doing reductions of Hamiltonian systems. Note that contactifications are not uniquely determined but the compact ones correspond to prequantization (or $S^1$-principal) bundles and are possible only for compact symplectic manifolds satisfying the Dirac quantization condition.

\medskip The organization of the paper is the following. For the purposes of reduction procedures, we start with an introduction to reductions of differential forms and basics on contact and symplectic manifolds. A section devoted to the geometry of Quantum Mechanics in finite dimensions is aimed to provide us an understanding of the set of quantum states as a convex body in the space of Hermitian operators and the manifold of pure states (complex projective space) as a coadjoint orbit of the unitary group. This leads to the standard examples of a contactification for complex projective spaces.

Further, we show, how to obtain coadjoint orbits in the process of Marsden-Weinstein-Meyer symplectic reduction, and describe where contactifications are hidden in the reduction.
The corresponding strategy is much more clear in the case of compact (e.g., unitary) Lie groups, whom we devote a separate section. Here, they appear obstructions for the existence of a compact contactification, completely equivalent to the Dirac quantization conditions in the theory of geometric quantization. In our approach, these obstructions have a simple topological sense.

The corresponding formalism is made explicit for unitary groups. Finally, we realize our main aim by constructing Lagrangians and the corresponding action functionals for Hamiltonian systems on compact symplectic manifolds by means of a contactification. It is illustrated with an example related to a magnetic monopole. We will work in the smooth category and all manifolds we consider are Hausdorff and paracompact.

\section{Reductions of differential forms}
Let us start with fixing some terminology and recalling some basic facts from differential geometry.

Suppose $\sD\subset\sT M$ is a distribution of rank $r$ (distributions are always smooth and regular in this paper) on a manifold $M$ of dimension $m$. We say that a (local) vector field $X$ on $M$ \emph{belongs to $\sD$} (and write $X\in\sD$) if $X_q\in\sD_q$ for all $q\in M$. The distribution $\sD$ is \emph{involutive} if the Lie bracket $[X,X']$ of any vector fields $X,X'\in\sD$ also belongs to $\sD$. According to the celebrated Frobenius Theorem, in this case \emph{integral submanifolds} of $\sD$, i.e., maximal connected immersed  submanifolds $N\subset M$ such that $\sT N\subset\sD$, have dimension $r$ and form a smooth \emph{foliation} $\cF(\sD)$ of $M$. For any point $q\in M$, there is its open neighbourhood $U$ equipped with local coordinates $(x^1,\dots,x^r,y^1,\dots,y^{m-r})$ such that the leaves of $\cF(\sD)\cap U$ are defined by $y=const\in\R^{m-r}$.

\mn The distribution $\sD$ we will call \emph{simple} if $\sD$ is involutive and the corresponding foliation $\cF(\sD)$ of $M$ is simple, i.e., there is a smooth manifold structure on the topological space $M/\sD=M/\cF(\sD)$ of the leaves of the foliation $\cF$ such that the canonical projection $p_\sD:M\to M/\sD$ is a surjective submersion (a smooth fibration). In this case, the coordinates $(y^j)$ mentioned above can be taken to be the pull-backs of some local coordinates $(u^j)$ around $p_\sD(q)\in M/\sD)$, $y^j=u^j\circ p_\sD$.
\begin{definition}\label{d1}
Let $\za$ be a $k$-form on a manifold $M$ and $v\in\sT M$. For an involutive distribution $\sD$ on $M$, the form $\za$ is \emph{$\sD$-invariant} if $\Ll_X \za=0$ for all $X\in\sD$.
Let us denote with $\si_v\za$ the (left) contraction of $\za$ with $v$ (the symbol $i$ we will need later to denote the imaginary root of $-1$).
The \emph{kernel} of $\za$, denoted $\ker(\za)$, is the subset of those vectors $v$ from $\sT M$ such that $\si_v\za=0$. The subset
$$\zq(\za)=\ker(\za)\cap\ker(\xd\za)\subset\sT M$$
we call, in turn, the \emph{characteristic set} of $\za$.
We call $\za$ \emph{regular} if $\zq(\za)$ is a (regular) distribution. In this case, we call $\zq(\za)$ the \emph{characteristic distribution} of $\za$. If the characteristic distribution consists of zero-vectors, we call $\za$ \emph{nondegenerate}.
\end{definition}
\no Note that $\ker(\za)$ is generally not a submanifold, nor a generalized (smooth) distribution (it is not locally generated by smooth vector fields), and if $\za$ is a closed, then $\zq(\za)=\ker(\za)$. Moreover,  simple but very useful observations in this context are the following.
\begin{lemma}\label{l2} Let $\za$ be a $k$ form on $M$ and $\sD$ be an involutive distribution on $M$. Then,
\begin{description}
\item{(a)} If a vector field $X$ belongs to $\zq(\za)$, then $\Ll_X\za=0$.
\item{(b)} The $k$-form $\za$ is $\sD$-invariant if and only if $\sD\subset\zq(\za)$.
\item{(c)} If $\za$ is regular, then $\zq(\za)$  is an involutive distribution and $\za$ is $\zq(\za)$-invariant.
\end{description}
\end{lemma}
\begin{proof}
(a) \ In view of the `magic' Cartan formula,
\be\label{mC}\Ll_X\za=\xd\si_X\za+\si_X\xd\za=0.\ee

\no (b) \ The $(\Leftarrow)$ part follows immediately from (a).

\no To prove $(\Rightarrow)$, let us notice that for any function $f$ on $M$ we have
$$\Ll_{fX}\za=\si_{fX}\,\xd\za+\xd\left(f\,\si_X\za\right)=f\Ll_X\za+\xd f\we\si_X\za.$$
If $X\in\sD$, then $fX\in\sD$, so $\xd f\we\si_X\za=0$ for all functions $f$ and thus $\si_X\za=0$.
But from (\ref{mC}) we get also $\si_X\xd\za=0$.

\mn (c) \ In view of (b), it is enough to show that $\zq(\za)$ is involutive.
Let $X,Y$ be vector fields belonging to $\zq(\za)=\ker(\za)\cap\ker(\xd\za)$.
Then, according to the well-know identity
$$\si_{[X,Y]}=\Ll_X\circ\si_Y-\si_Y\circ\Ll_X,$$
we get
$$\si_{[X,Y]}\za=\Ll_X\si_Y\za-\si_Y\left(\xd\,\si_X\za+\si_X\xd\za\right)=0.$$
Replacing now $\za$ with $\xd\za$ in the left-had-side of the above identity, we get also
$\si_{[X,Y]}\xd\za=0$.

\end{proof}
\begin{corollary}\label{formred}
If $\za$ is a differential $k$-form on $M$ which is invariant with respect to a simple distribution $\sD\subset\sT M$, then there is a unique $k$-form $\za/\sD$ on the manifold $M/\sD$ such that
$p_\sD^*(\za/\sD)=\za$. Moreover, in the case when $\za$ is regular and $\zq(\za)$ is a simple
distribution, the $k$-form $\za_{red}=\za/\zq(\za)$ on $M_{red}=M/\zq(\za)$ is nondegenerate.
\end{corollary}
\begin{proof}
Let us take $y_0\in M/\sD$ and $q\in M$ such that $p_\sD(q)=y_0$. Let $(y^j)$ be local coordinates around $y_0$ in $M/\sD$. Denoting their pull-backs $y^j\circ p_\sD$, with some abuse of notation, also $y^j$, we can pick-up $x^i$ so that $(x^i,y^j)$ are local coordinates around $q$ in $M$. In particular, $\pa_{x^i}\in\sD$. Hence, according to Lemma \ref{l2}, $\si_{\pa_{x^i}}\za=0$ for all $i$, so there are no $\xd x^i$ in $\za$. Consequently, we can write $\za$ in the form
$\za=g_I\xd y^I$, where $I=(i_1,\dots,i_{m-r})$, $i_s=0,1$, are multi-indices, and
$$\xd y^I=(\xd y^1)^{i_1}\we\cdots\we(\xd y^{m-r})^{i_{m-r}},$$
with $(\xd y^j)^0=1$. We get further
$$ 0=\Ll_{\pa_{x^i}}\za=\frac{\pa g_I}{\pa x^i}\xd y^I.
$$
Hence, the coefficients $g_I$ do not depend on variables $(x^i)$, $g_I=g_I(y)$, so $\za$ can be written with the use of coordinates $(y^j)$ only and therefore it can be treated as the pull-back of a unique $k$-form $\za/\sD$, which in coordinates $(y^j)$ looks formally exactly as $\za$, $\za/\sD=g_I(y)\xd y^I$. Of course, if $\sD=\zq(\za)$, then the reduction kills degeneration and $\za/\zq(\za)$ is nondegenerate.

\end{proof}
\begin{definition} Let $M$ be a manifold of dimension $m$.

\no\begin{description}\label{forms}
\item{(a)} Regular $k$-forms on $M$ with the characteristic foliation being simple we call \emph{simple}, and the foliation $\cF(\za)=\cF(\zq(\za))$ of $M$ we cal the \emph{characteristic foliation} of $\za$. If $\za$ is simple, then the pair $(M_{red},\za_{red})$ we call the \emph{regular reduction} of $(M,\za)$.
\item{(b)} A submanifold $\si_N:N\to M$ of a manifold $M$ equipped with a $k$ form $\za$ we call \emph{regular} if the restriction $\za_N=\si^*(\za)$ of $\za$ to $N$ is a regular form. If $\za_N$ is simple, we speak about a \emph{simple} submanifold $N$ of $(M,\za)$.
\item{(c)} A 2-form $\zw$ on $M$ we call \emph{presymplectic of rank $2r$} it $\zw$ is closed and $\zq(\zw)=\ker(\zw)$ is of rank $m-2r$. If $\zw$ is additionally nondegenerate, i.e., $2r=m$, we call $\zw$ a \emph{symplectic} form. A manifold equipped with a (pre)symplectic form we call \emph{(pre)symplectic}.
\item{(d)} A submanifold $\si_N:N\hookrightarrow M$ of a symplectic manifold $(M,\zw)$ we call \emph{presymplectic} if the restriction $\zw_N=\zw\,\big|_N=\si_N^*(\zw)$ of $\zw$ to $N$ is regular, thus presymplectic. A presymplectic manifolds $(M,\zw)$ equipped with a \emph{potential}, i.e., a 1-form $\zvy$ such that $\xd\zvy=\zw$, we call \emph{exact}.
\end{description}
\end{definition}
\no Note that the rank of a 2-form is even at every point. Above, we used the definition of a \emph{presymplectic form} as it was introduced by Souriau.
Some authors consider presymplectic forms simply as closed 2-forms, which is too weak for our purposes. Even in this case one assumes in applications the regularity of $\zw$.

\mn With simple regular submanifolds $N$ of $(M,\za)$ we can associate \emph{differential form reduction} which follows immediately from Corollary \ref{formred}.
\begin{theorem}[Regular reduction of differential forms]\label{dfr}
Let $\za$ be a $k$-form on a manifold $M$, and let $\si_N:N\hookrightarrow M$ be simple submanifold of $(M,\za)$. Then the $k$-form $\za_N=\si_N^*(\za)$ on $N$ is simple and there is a nondegenerate $k$-form $\za^N_{red}=(\za_N)_{red}$ on $N_{red}^\za=N/\zq(\za_N)$ such that $p_\za(N)^*(\za^N_{red})=\za_N$, where $p_\za(N):N\to N_{red}^\za$ is the canonical surjective submersion.
\end{theorem}
\no The procedure of passing from $(M,\za)$ to $(N_{red}^\za,\za^N_{red})$ we call \emph{$\za$-reduction} of $N$. If $\za$ is symplectic, the $\za$-reduction is traditionally called a \emph{symplectic reduction} (cf. \cite{Libermann:1987}). In this case, $\za_N$ is a presymplectic form on $N$.
For presymplectic forms we have the following.
\begin{theorem}[Presymplectic Darboux Theorem]
If $\zw$ is a presymplectic form of rank $2r$ on a manifold $M$ of dimension $m$, then around every point of $M$ there are local coordinates
$$(p_1,\dots,p_r,q^1,\dots,q^r,z^1,\dots,z^{m-2r})$$
in which $\zw$ reads
$$\zw=\xd p_i\we\xd q^i.$$
\end{theorem}

\section{Contactifications of symplectic manifolds}  Consider now a distribution $C\subset \sT M$ being a \emph{field of hyperplanes} on $M$, i.e., a distribution with rank $(m-1)$. Such a distribution is, at least locally, the kernel of a nonvanishing 1-form $\zh$ on $M$, i.e., $C=\ker(\zh)$. Of course, the 1-form $\zh$ is determined only up to conformal equivalence. Denote the line bundle $\sT M/C\to M$ with $L^C$ and let $\zt^C:\sT M\to L^C$ be the canonical projection. The map $\zn^C:C\ti_MC\to L^C$ which for vector fields $X,Y$ on $M$, taking values in $C$, reads $\zn^C(X,Y)=\zt^C([X,Y])$, is a well-defined skew-symmetric 2-form on $C$ with values in the line bundle $L^C$. Indeed, if $f:M\to\R$ is a function on $M$, then
\beas&&\zn^C\left(X,fY\right)=\zt^C([X,fY])=\zt^C\left(f[X,Y]+X(f)Y\right)\\
&&=f\zn^C(X,Y)+X(f)\zt^C(Y)=f\zn^C(X,Y)\,.\eeas
\begin{definition} A hyperplane field $C\subset\sT M$ we call a \emph{contact structure} if the 2-form $\zn^C$ on $C$ is nondegenerate. Manifolds equipped with a contact structure we will call \emph{contact manifolds}. Any nonvanishing (local) 1-form which determines a contact structure $C$ we call a \emph{contact form}.
If a globally defined contact form $\zh$ is chosen, then the corresponding contact structure $C=\ker(\zh)$ we call \emph{trivial} or \emph{co-oriented}. If such a contact form exists, we call the contact structure \emph{co-orientable} (trivializable).
\end{definition}
Of course, all 1-forms $f\zh$, $f\ne 0$, i.e., forms conformally equivalent to a contact form $\zh$, are also contact forms. In this paper we will deal only with trivial contact structures $(M,\zh)$, i.e., manifolds $M$ equipped with a globally defined contact form $\zh$.
\begin{remark} As integrability of a distribution is equivalent to its involutivity, the property that the bracket $[X,Y]$ is nondegenerate on $C$ is sometimes expressed as \emph{maximal non-integrability}.
\end{remark}
\begin{theorem}[Contact Darboux Theorem] Let $\zh$ be a  1-form on a manifold $M$ of dimension $(2n+1)$. Then $\zh$ is a contact form if and only if $\zW=\zh\we(\xd\zh)^n$ is a nonvanishing volume form. In such a case, around every point of $M$ there are local coordinates $(z,p_i,q^i)$, $i=1,\dots,n$, in which the contact form $\zh$ reads
\be\label{Dc} \zh=\xd z-p_i\,\xd q^i.\ee
\end{theorem}
\no It is clear from the above theorem that contact forms are nondegenerate in the sense of Definition \ref{d1}. A sort of a converse of this statement is also true.
\begin{proposition}
Let $\zh$ be a nondegenerate 1-form on a manifold $M$ of dimension $m$. Then,
\begin{description}
\item{(a)} the form $\zh$ is a contact form if $m$ is odd;
\item{(b)} the form $\zw=\xd\zh$ is a symplectic form if $m$ is even.
\end{description}
In other words, manifolds equipped with a nondegenerate 1-form are either co-oriented contact manifolds or exact symplectic manifolds, depending on the parity of the dimension of $M$.
\end{proposition}
\begin{proof}
(a) \ As $m=2r+1$, the rank of $\xd\zh$ must be everywhere $2r$. Indeed, being even this rank is at most $2r$. If at a point $x\in M$ this rank is $\le 2r-2$, then $\dim(\ker(\xd\zh(x)))\ge 3$, so the dimension of the intersection of $\ker(\xd\zh(x))$ with $\ker(\zh(x))$ is at least 2; a contradiction. Since $\ker(\xd\zh)$ is 1-dimensional and trivially intersects $\ker(\zh)$, the 2-form $\xd\zh$ is nondegenerate on $\ker(\zh)$, thus $\zh$ is a contact form.

\mn (b) \ Suppose now that $m=2r$. We should show that $\ker(\xd\zh)$ is trivial at all points.
Indeed, if $\ker(\xd\zh)$ is not trivial at a point $x$, then its dimension at $x$ is at least 2.
But the dimension of $\ker(\zh(x))$ is at least $(m-1)$, thus the intersection of these kernels
is at least 1 at $x$; a contradiction.

\end{proof}
\begin{corollary}[Contact reduction]\label{cc1} If $\zh$ is a 1-form on a manifold $M$ and $N$ is a simple submanifold of $(M,\zh)$, then the 1-form $\zh_{red}^N$ obtained by the $\zh$-reduction of $N$ (Theorem \ref{dfr}) is a contact form if the dimension of $N_{red}^\zh$ is odd, and a symplectic
potential, i.e., $\xd(\zh_{red}^N)$ is symplectic, if this dimension is even.
\end{corollary}
\noindent The following is a particular case of the well-known Poisson-to-Jacobi reduction (see e.g. \cite{Grabowski:2004}).
\begin{proposition}[Symplectic-to-contact reduction]
Let $(M,\zw)$ be a symplectic manifold and $N$ be its closed submanifold in $M$ of codimension 1. If $X$ is a vector field defined in a neighbourhood of $N$ which is transversal to $N$ and satisfies $\Ll_X\zw=\zw$, then the 1-form $\zh$ on $N$ defined by
$$\zh_x(Y)=\zw_x(X,Y), \ \text{where}\ x\in N,\ Y\in\sT_xN,$$
is a contact form on $N$.
\end{proposition}
\begin{proof}
The 1-form $\zh$ is clearly nonvanishing and
$$\xd\zh=\xd(\si_X\zw\,\big|_N)=(\xd\si_X\zw)\,\big|_N=(\Ll_X\zw)\,\big|_N=\zw\,\big|_N.$$
Hence $\xd\zh$ is of corank 1, so it suffices to show that $$\zq(\zh)=\ker(\zh)\cap\ker(\xd\zh)=\{0\}.$$
But if $Y_0\in\zq(\zh)$, then $\zw(X,Y_0)=0$ and $\zw(Y,Y_0)=0$ for all vector fields $Y$ on $N$.
Hence, $Y_0\in\ker(\zw)=\{0\}$.

\end{proof}
\begin{remark} Any contact form $\zh$ on $M$ determines uniquely a vector field $\cR$ on $M$, called the \emph{Reeb vector field}, which is characterized by the equations
$$i_\cR\zh=1\quad \text{and}\quad i_\cR\xd\zh=0.$$
For the contact form (\ref{Dc}) we have $\cR=\pa_z$.
\end{remark}
It is obvious that any co-oriented contact manifold $(M,\zh)$ of dimension $(2n+1)$ is automatically presymplectic with the exact presymplectic form $\xd\zh$ of rank $2n$. In this case the involutive distribution $\ker(\xd\zh)$ is generated by the Reeb vector field $\cR$. The contact structures in connection to physics and Hamiltonian mechanics became recently a subject of intensive studies (see, e.g., \cite{Bravetti:2017,Bruce:2017,Ciaglia:2018,deLeon:2019,Grabowska:2022a,Grabowska:2023,Grabowski:2013}). For a purely geometric approach to contact Hamiltonian mechanics which serves for general contact structures (not necessarily co-oriented) we refer to \cite{Grabowska:2022}.
\begin{definition}
If the distribution $\la\cR\ran=\ker(\xd\zh)$ is simple, we call the contact manifold $(M,\zh)$ \emph{regular}.
\end{definition}
\no The following is obvious.
\begin{proposition}[Contact-to-symplectic reduction]
If $(M,\zh)$ is a regular contact manifold, then $\zw=\xd\zh/\la\cR\ran$ is a symplectic form on the reduced manifold $M_{red}=M/\la\cR\ran$. Moreover, $\zw$ is uniquely characterized as a 2-form on $M$ such that $p^*(\zw)=\xd\zh$, where $p:M\to M_{red}$ is the canonical surjective submersion.
\end{proposition}
\begin{definition} The procedure of passing from $(M,\zh)$ to $(M_{red},\zw)$ we call \emph{contact-to-symplectic reduction}, and the contact manifold $(M,\zh)$ -- a \emph{contactification} of the symplectic manifold $(M_{red},\zw)$.
\end{definition}
\begin{example}\label{exx}
It is easy to see that any exact symplectic manifold $(N,\zw)$, where
$\zw=\xd\zvy$, admits a canonical contactification. We just put $M=N\ti\R$ and $\zh(x,t)=\xd t+\zvy(x)$. It is easy to see that the Reeb vector field is $\cR=\pa_t$ and $\xd\zh(x,t)=\zw(x)$. This construction was called a \emph{contactification} in \cite[Appendix 4]{Arnold:1989}.
\end{example}
The situation is generally much more complicated for symplectic manifolds $(N,\zw)$ such that $\zw$ is not exact. It is always the case if $N$ is compact, so the canonical Fubini-Study symplectic forms on the complex projective spaces $\C\P^n$ are never exact. Note also that removing, say, one point from a contactification, we get a new contactification, that shows that contactifications of a given symplectic manifold are never unique. On the other hand, we have the following.
\begin{proposition}
Every symplectic manifold $(N,\zw)$ admits a contactification.
\end{proposition}
\begin{proof}
Consider a locally finite atlas $\{(U_n,\zf_n)\}_{n\in\N}$ on $N$ with contractible open subsets $U_n$. Let $\zvy_n$ be a potential for $\zw$ on $U_n$, $\xd\zvy_m=\zw\,\big|_{U_n}$, and equip $V_n=U_n\ti(n,n+1)\subset N\ti\R$ with the contact form $\zh_n=\xd t+\zvy_n$, where $t$ is the standard coordinate in $\R$. The manifold $M=\cup_nV_n\subset N\ti\R$ with the contact form $\zh$ such that $\zh\,\big|_{V_n}=\zh_n$ is a contactification of $(N,\zw)$. Indeed, the Reeb vector field is $\pa_t$ on $N\ti\R$ restricted to $M$, so $(M,\zh)$ is regular and $M\to M_{red}$ is the canonical projection $p:N\ti\R\to N$ restricted to $M$. Hence $M_{red}=N$ and $\xd\zh=p^*(\zw)$.
\end{proof}
The problem with the above contactification is that the contact manifold is non-connected, so it is not actually `global'. Therefore, it makes sense to consider only connected contactifications of connected symplectic manifolds which makes the problem of existence of a contactification really nontrivial. For instance, compact contactifications are quire rare, as in this case the Reeb vector field is complete.
\begin{definition}
If the Reeb vector field $\cR$ is a complete, we call the contact manifold $(M,\zh)$ \emph{complete}.
If $\cR$ is the generator of a principal action of a Lie group $\cG=\U(1)$ or $\cG=\R$ on $M$, we call the contact manifold $(M,\zh)$ \emph{principal}.
\end{definition}
\no Of course, any principal contact manifold is regular and complete. The main result in \cite{Grabowska:2023} states the converse.
\begin{theorem}[\cite{Grabowska:2023}]
Any connected regular and complete contact manifold $(M,\zh)$ is principal.
Moreover, if the structure group of the principal bundle $p:M\to N=M_{red}$ is $\U(1)$ and $2\pi\hbar$ is the minimal period of the flow of the Reeb vector field $\cR$, then the symplectic form $\zw$ on $N$ is $\Z_\hbar$-integral, where $\Z_\hbar=2\pi\hbar\cdot\Z$, i.e.,
$$\big[\zw/2\pi\hbar\big]\in\sH^2(N,\Z).$$
In the case of the $\R$-principal bundle, in turn, the symplectic form $\zw$ is exact, $\zw=\xd\zvy$ for a 1-form $\zvy$, and $(M\simeq N\ti\R,\zh)$ is the standard contactification of an exact symplectic manifold like in Example \ref{exx}.
\end{theorem}
\no We will explain closer the concept of $\Z_\hbar$-integrality in Section \ref{A}.
\begin{remark}
If $(M,\zh)$ is regular and compact contact manifold, then it is automatically complete, the structure group is $\U(1)$, and we get the celebrated Boothby-Wang theorem \cite{Boothby:1958} as a particular example. Note, however, that the original proof in \cite{Boothby:1958} contains a substantial gap.
\end{remark}
\begin{corollary}
Any compact contactification of a compact symplectic manifold is canonically an $\U(1)$-principal bundle with respect to an $\U(1)$-action induced by the flow of the Reeb vector field.
\end{corollary}
\section{Geometry of Quantum Mechanics}
The complex projective space $\C\P^{n-1}$ consists of pure states on the Hilbert space $\C^n$ and is a minimal coadjoint orbit of the unitary group $\U(n)$.
The standard Hermitian product $\la x\,|\,y\ran$ on $\C^n$ is, by convention, $\C$-linear with respect to $y$ and anti-linear with respect to $x$. The unitary group $\U(n)$ acts on $\C^n$
preserving this Hermitian product and consists of those complex matrices $U\in\gl(n;\C)$ which satisfy $UU^\dag=I$, where $U^\dag$ is the Hermitian conjugate of $U$, i.e.
$$
\la Ux\,|\,y\ran=\la x\,|\,\,U^\dag y\ran.
$$
The Lie algebra $\u(n)$ of $\U(n)$ is the real vector space of anti-Hermitian matrices $T^\dag=-T$ with the commutator bracket $[T,T']=TT'-T'T$. In what follows, we will write simply $\U$ for $\U(n)$ and $\u$ for $\u(n)$.

\mn The geometric approach to Quantum Mechanics is based on the observation that the
`realification' $\C^n_\R=\R^n\oplus i\R^n=\R^{2n}$ of $\C^n$ is a K\"ahler manifold $(\R^{2n},g,\zw)$, with
a Riemannian metric $g$ and a symplectic form $\zw$ defined by
$$ g(x,y)+i\cdot\zw(x,y)=\la x\,|\,y\ran.$$
In other words, $g(x,y)=\re\bkk{x}{y}$ and $\zw(x,y)=\im\bkk{x}{y}$, where $\re$ and $\im$ denote the real and the imaginary part of a complex number, respectively. The standard orthonormal basis $(e_k)$ of $\C^n$ induces global coordinates
$(q^k,p_l)$, $k,l=1,\dots,n$, on $\R^{2n}$ by
$$\la e_k\,|\,x\ran=(q^k+i\cdot p_k)(x).$$
In these coordinates, the vector $\pa_{q^k}$ represents $e_k$ and $\pa_{p_k}$ represents
$i\cdot e_k$. Hence, the Riemannian tensor reads
$$g=\sum_k\left(\xd q^k\ot\xd q^k+\xd p_k\ot\xd p_k\right)$$
and the symplectic form
$$\zw=\sum_k\xd q^k\we \xd p_k=\sum_k\left(\xd q^k\ot\xd p_k-\xd p_k\ot\xd q^k\right).$$
The 1-form
\be\label{zvy}\zvy=\half\sum_k\big(q^k\xd p_k-p_k\xd q^k\big)\ee
we will call the \emph{Liouville 1-form}. It is clearly a potential for $\zw$, $\xd\zvy=\zw$.
The unitary group $\U$ acts $\R$-linearly on $\R^{2n}\simeq\C^n$ and this action is simultaneously an isometry and a symplectomorphism,
$$g(x,y)+i\cdot\zw(x,y)=\la x\,|\,y\ran=\la Ux\,|\,Uy\ran=g(Ux,Uy)+i\cdot\zw(Ux,Uy).$$
Moreover, the Liouville 1-form $\zvy$ is $\U$-invariant.

\medskip One important convention we want to introduce following \cite{Grabowski:2005} is that we will
identify the real vector space of Hermitian operators $A=A^\dag$ with the dual
$\u^*$ of the real Lie algebra $\u$, according to the
pairing between Hermitian $A\in \u^*$ and anti-Hermitian $T\in \u$ operators,
$$\la A,T\ran=\frac{i}{2}\cdot\tr(A T).$$
The multiplication by $i$ establishes further vector space
isomorphisms
\be\label{iso}\u\ni T\mapsto iT\in \u^*\,,\quad \u^*\ni A\mapsto -iA\in \u,\ee which identify the adjoint $\Ad_U(T)=U T U^\dag$ with the coadjoint
$$(\Ad^*)_U(A)=(\Ad_{U^{-1}})^*(A)=U A U^\dag$$
action of the group $\U$. Under these isomorphisms, $\u^*$ and $\u$
become equipped with the scalar products
\be\label{metric1}\la A,B\ran_{u^*}=\frac{1}{2}\,\tr(AB)\,\quad \la T,T'\ran_{u}=-\frac{1}{2}\,\tr(TT'),
\ee
respectively, and a Lie bracket
\be\label{bra}[A,B]_{u^*}=-i[A,B],\ee
where $[A,B]=AB-BA$ is the commutator bracket. It is easy to see that the fundamental vector field
of the coadjoint action associated with $iA\in \u$ takes at $\zm\in\us$ the value $-i[A,\zm]=[A,\zm]_{u^*}$.
The scalar products (\ref{metric1}), the bracket (\ref{bra}) and the commutator bracket in $\u$ are invariant with respect to the coadjoint/adjoint action of the unitary group.
\begin{remark} The set of \emph{quantum states} can be defined as the subset $\mathbf{QS}$ in $\us$ consisting of those $\zm\in\us$ which are positive, $\bkk{x}{\zm x}\ge 0$, and have trace $1$, $\tr(\zm)=1$. It is invariant with respect to the coadjoint $\U$-action on $\us$. The subset $\mathbf{QS}$ in $\us$ is not a manifold if $\dim_\C(\cH)\ne 2$, but it is canonically stratified by submanifolds $\mathbf{QS}^k$, $k=1,\dots,\dim_\C(\cH)$, of $\us$, where $\mathbf{QS}^k$ consists of quantum states with rank $k$ (see \cite{Grabowski:2005}). Quantum states of rank $1$ we call \emph{pure}, and states of rank $k$ \emph{mixed}. Quantum states form a compact convex set in $\us$ whose extremal points are exactly pure states. In other words, any quantum state is a convex combination of pure states. These convex combinations are called in physics \emph{quantum mixtures}.
Note additionally that there is a finer stratification of $\mathbf{QS}$ by manifolds than the stratification by rank, namely the stratification by coadjoint orbits. Two quantum states belong to the same coadjoint orbit if and only if they have the same spectrum. For quantum states this spectrum consists of non-negative reals.
\end{remark}
\no The Lie algebra structure on $\u$ induces a linear Poisson tensor (Kostant-Kirillov-Souriau tensor) $\zL$  on $\us$ inducing a Poisson bracket $\{\cdot,\cdot\}_\zL$, which on linear functions
$$F_T(A)=\la T,A\ran=\frac{i}{2}\,\tr(TA),$$
with $T\in \u$, takes the form
$$
\{F_T,F_{T'}\}_\zL=F_{[T,T']}.
$$
The Hamiltonian vector field $X_{F_T}$ for this Poisson structure reads
$$X_{F_T}(\zm)=-\ad^*_T(\zm)=[T,\zm]=T\circ\zm-\zm\circ T,$$
so
$$\{F_T,F_{T'}\}_\zL=X_{F_T}(F_{T'})=\Bk{T'}{-\ad_T^*(\zm)}=\Bk{\zm}{[T,T']}.$$
It is well known that the characteristic distribution of this Poisson structure consists of coadjoint orbits of $\U$. The canonical symplectic structure $\zw^\cO$ on the $\U$-orbit $\cO$ is called the \emph{Kostant-Kirillov-Souriau structure} (KKS structure in short).

Vectors tangent to the orbit $\cO$ at $\zm$ are therefore vectors of the Hamiltonian vector fields at $\zm$, thus have the form $\ad^*_T(\zm)=[\zm,T]$, where $T\in\u$. Actually, we can identify $\sT_\zm\cO$ with $\u/\u_\zm$, where $\u_\zm$ is the Lie subalgebra in $\u$ of those $T$ for which $\ad^*_T(\zm)=[\zm,T]=0$. This is the Lie algebra of the subgroup $\mathrm{U}_\zm$ of $\U$ being the isotropy subgroup of $\zm\in\us$ with respect to the coadjoint action. Consequently, the symplectic form $\zw^\cO$ can be defined \emph{via} (this, in fact, is valid for any Lie group)
\be\label{KKSf}
\zw^\cO_\zm\big(\ad^*_X(\zm),\ad_Y^*(\zm)\big)=\Bk{\zm}{[X,Y]}.
\ee
\section{Contactifications of complex projective spaces}
Let us go back to the symplectic $\U$-action on $\R^{2n}$. It is easy to see that the map
$$J:\C^n\to \us,\quad J(x)=\zr_x=\kb{x}{x},$$
where we use the Dirac notation $\kb{x}{x}(y)=\bk{x}{y}x$, is an equivariant moment map for this action.
In particular, the momentum map image of the $(2n-1)$-dimensional sphere $S^{2n-1}=\{ x\in\C^n:\Vert x\Vert=1\}$ is the manifold
$$\cO=\big\{\kb{x}{x}:\,\Vert x\Vert=1\big\}$$
of pure states. This is a coadjoint orbit, so it posses the KKS symplectic structure (\ref{KKSf}). Hence, vectors tangent to the submanifold ${\cO}$ in $\us$ at $\zr_x$ have the form
$$[T,\zr_x]=\kb{x}{Tx}+\kb{Tx}{x}$$
for $T\in\u$. Writing $y=Tx$, we get that these tangent vectors are of the form
$$A_{x,y}=\kb{x}{y}+\kb{y}{x}\in\us.$$
But $T$ is anti-Hermitian, so $y\in\C^n$ cannot be arbitrary. Indeed,
$$0=\bkk{Tx}{x}+\bkk{x}{Tx}=2\re\bkk{x}{Tx}=2g(x,y).
$$
so $y\in\R^{2n}$ must be tangent to the unit sphere, $y\in\sT_xS^{2n-1}$, and it is easy to see that
every vector tangent to ${\cO}$ at $\zr_x$ is of this form. But
$$\sT_xJ(y)=\kb{x}{y}+\kb{y}{x},$$
so
$$J_0=J\big|_{S^{2n-1}}:S^{2n-1}\to{\cO}$$
is a surjective submersion (fibration). It is easy to see that the kernel of $\sT_xJ_0$ contains  vectors
$ix$, so the kernel of $\sT J_0$ consists of vectors tangent to orbits of the subgroup
$$\U(1)=\big\{e^{it}\cdot I\,|\, t\in\R\big\}=\big\{zI\,|\, z\in\C,\ \vert z\vert=1\big\}$$
of $\U$.
This shows that $S^{2n-1}$ is a principal bundle over ${\cO}$ with the structure group $\U(1)=S^1$, thus ${\cO}=S^{2n-1}/S^1$. Since, in turn, $S^{2n-1}=\big(\C^n\big)^\ti/\R_+$, where the multiplicative group $\R_+$ of positive reals acts freely on $\big(\C^n\big)^\ti=\C^n\setminus\{0\}$ by multiplication, we get finally that ${\cO}$ is the complex projective space $\C\P^{n-1}$ of real dimension $(2n-2)$,
$${\cO}=S^{2n-1}/S^1=\Big(\big(\C^n\big)^\ti/\R_+\Big)/S^1=
\big(\C^n\big)^\ti/\C^\ti=\C\P^{n-1}.$$
Here, the multiplicative group of non-zero complex numbers $\C^\ti$ acts freely and properly on $\C^n$ by the complex multiplication. Note that for $n=2$ we get the celebrated \emph{Hopf fibration}. We will analyse closer this fiber bundle in Section \ref{B}.

\mn We already know that ${\cO}$ is canonically a symplectic manifold and the KKS symplectic form $\zw^\cO$ on ${\cO}$ reads (\ref{KKSf})
\be\label{KS}\zw^\cO_{\zr_x}(A_{x,y},A_{x,y'})=\zw^\cO_{\zr_x}\big([\zr_x,T_{x,y}],[\zr_x,T_{x,y'}]\big)
=\Bk{\zr_x}{[T_{x,y},T_{x,y'}]}=\frac{i}{2}\,\tr\big(\zr_x\circ[T_{x,y},T_{x,y'}]\big).
\ee
Here, $T_{x,y}$ is any $T\in\u$ such that $T(x)=y$. But for $\Vert x\Vert=1$ and $\bk{x}{y}=0$ we have
\beas&\tr\big(\zr_x\circ[T_{x,y},T_{x,y'}]\big)=\Bk{x}{[T_{x,y},T_{x,y'}](x)}=
\bkk{T_{x,y'}(x)}{T_{x,y}(x)}-\bkk{T_{x,y}(x)}{T_{x,y'}(x)}\\
&=\bkk{y'}{y}-\bkk{y}{y'}=2i\im\bkk{y'}{y}.
\eeas
This, combined with (\ref{KS}), yields
$$\zw^\cO_{\zr_x}\big(A_{x,y},A_{x,y'}\big)=\im\bkk{y}{y'}=\zw(y,y').$$
Since $A_{x,y}=(\sT_xJ_0)(y)$ for $y\in\sT_xS^{2n-1}$, the above means that the KKS symplectic form $\zw^\cO$ on ${\cO}$ is the symplectic reduction of the coisotropic submanifold $S^{2n-1}\subset\R^{2n}$. In other words, $J_0^*(\zw^\cO)=\zw_0$, where $\zw_0=\zw\,\big|_{S^{2n-1}}$. Of course, $\zw_0=\xd\zvy_0$, where $\zvy_0=\zvy\,\big|_{S^{2n-1}}$, and the characteristic foliation of $\zw_0$ consists of the orbits of the $\U(1)$-action. It is easy to see that the generator $X(x)=ix$ of this action satisfies $i_X\zvy_0=1/2$ and $i_X\xd\zvy_0=i_X\zw_0=0$, so $(S^{2n-1},\zvy_0)$ is a contactification of the coadjoint orbit $(\cO,\zw^\cO)$, with the Reeb vector field $\cR=2X$. This way we get the following.
\begin{theorem}\label{cpn}
The Liouville 1-form $\zvy$ on $\C^n$ (see (\ref{zvy})) restricted to the unit sphere $S^{2n-1}\subset\C^n$ is a contact form $\zvy_0=\zvy\,\big|_{S^{2n-1}}$ and the characteristic foliation of
$$\zw_0=\xd\zvy_0=\zw\,\big|_{S^{2n-1}}$$
consists of orbits of the canonical principal $\U(1)$-action on the sphere. Moreover, $S^{2n-1}/\U(1)=\C\P^{n-1}$ and $(S^{2n-1},\zvy_0)$ is a contactification of $\C\P^{n-1}$ equipped with its KKS-symplectic form $\zw^{\C\P^{n-1}}$, i.e., $\xd\zvy_0=\zp^*\Big(\zw^{\C\P^{n-1}}\Big)$, where
$$\zp:S^{2n-1}\to S^{2n-1}/\U(1)=\C\P^{n-1}$$
is the canonical projection.
\end{theorem}
\section{Coadjoint orbits \emph{via} Marsden-Weinstein reduction}
In this section we will show that coadjoint orbits of a Lie group $G$ can be obtained from the Marsden-Weinstein-Meyer symplectic reduction of $\sT^*G$. This result is well known (see e.g. \cite[Theorem 6.2.2]{Ortega:2004}, but to make our presentation self-contained (and due to our needs in the sequel) we will sketch a short proof here. All groups we consider are real Lie groups.

Let us assume initially for simplicity that $G$ is connected. Our starting point is the well-known observation that the tangent bundle $\sT G$ is trivializable as a vector bundle. We have, in principle, two canonical trivializations $\zl,\zr:\sT G\to G\ti\g$, associated with the left $L_g(h)=gh$ and the right $R_g(h)=hg$ translations in $G$, respectively. Namely, for $v_g\in\sT_gG$, we have
\be\label{triv}
\zl(v_g)=\big(g,\sT L_{g}^{-1}(v_g)\big)\,,\quad \zr(v_g)=\big(g,\sT R_{g}^{-1}(v_g)\big).
\ee
Here, of course, $\g$ denotes the Lie algebra of $G$. The `coordinates' associated with the \emph{left trivialization} $\zl$ are sometimes called \emph{body coordinates}, those associated with the \emph{right trivialization} $\zr$ are called \emph{space coordinates}. The corresponding left and right trivializations $\bar\zl,\bar\zr:\sT^*G\to G\ti\g^*$ of the cotangent bundle $\sT^*G$ take the form
\be\label{triv*}
\bar\zl(\za_g)=\big(g,(\sT L_{g})^*(\za_g)\big)\,,\quad \bar\zr(\za_g)=\big(g,(\sT R_{g})^*(\za_g)\big).
\ee
The left action of $G$ on itself by the left translations can be lifted canonically to the corresponding actions of $G$ on $\sT G$ and $\sT^*G$,
\beas &\L:G\ti \sT G\to\sT G\,,\ \L(g,v_h)=\L_g(v_h)=\sT L_g(v_h)\,\\
&\hL:G\ti \sT^*G\to\sT^*G\,,\ \hL(g,\za_h)=\hL_g(\za_h)={(\sT L_{g^{-1}})}^*(\za_h).
\eeas
In the space coordinates these actions look like
$$
\L_g(h,v)=(gh,\Ad_g(v))\,,\quad \hL_g(h,\zm)=(gh,\Ad^*_g(\zm)),
$$
where we denoted $\Ad^*_g={(\Ad_{g^{-1}})}^*$, so $\Ad^*$ is a representation of $G$ called the {coadjoint representation}. In the body coordinates it reads
$$\L_g(h,v)=(gh,v)\,,\quad \hL_g(h,\zm)=(gh,\zm),
$$
so the $G$ action is free and proper.
In what follows, we will work exclusively with the space coordinates. A useful reference could be
\cite[Ch. 4.5]{Abraham:1978} (see also \cite[Ch. 6.2]{Ortega:2004}), but we should stress that, due to differences of conventions, some signs in \cite[Ch. 4.5]{Abraham:1978} are different. Of course, all this is a question of convention which form we consider as the canonical symplectic form on $\cO$, similarly to the choice of convention for the sign in the definition of the Hamiltonian vector field, as well as deciding what is the canonical symplectic form $\zw$ on $\sT^*G$; e.g., Abraham-Marsden \cite{Abraham:1978}) take $\zw=-\xd\zvy$.
For us, in the space coordinates the Liouville 1-form $\zvy$ and the canonical symplectic form $\zw=\xd\zvy$ on $\sT^*G$ read simply
\bea\label{Liou}
&\zvy_{(g,\zm)}(v,\zs)=\bk{\zm}{v}\,,\\
&\zw_{(g,\zm)}\big((v,\zs),(v',\zs')\big)
=\bk{\zs}{v'}-\bk{\zs'}{v}+\bk{\zm}{[v,v']}\,,\label{csymp}
\eea
where
$$ (v,\zs),(v',\zs')\in\g\ti\g^*\overset{\zr}{=}
\sT_gG\ti\g^*=\sT_{(g,\zm)}(G\ti\g^*)\overset{\bar\zr}{=}\sT_{(g,\zm)}\sT^*G.
$$
First, let us notice that
$$\sT\hL_h\big((g,\zm),(v,\zs)\big)=\big((hg,\Ad^*_h(\zm)),(\Ad_h(v),\Ad^*_h(\zs)\big),$$
that implies that $\zvy$ is $G$-invariant. Indeed,
\beas &\hL^*_h(\zvy)_{(g,\zm)}(v,\zs)=\zvy_{\hL_h(g,\zm)}(\sT_{(g,\zm)}(v,\zs))=
\zvy_{(hg,\Ad^*_h(\zm))}(\Ad_h(v),\Ad^*_h(\zs))\\
&=\bk{\Ad^*_h(\zm)}{\Ad_h(v)}=\bk{\zm}{\Ad_h^{-1}\Ad_h(v)}=\bk{\zm}{v}=\zvy_{(g,\zm)}(v,\zs).
\eeas
Since $\zw=\xd\zvy$, the $G$-action on $\sT^*G$ is Hamiltonian, the canonical Hamiltonian associated with $v\in\g$ is $H_v(g,\zm)=\bk{\zm}{v}$, and the canonical equivariant moment map in the space coordinates reads simply
$$J:\sT^*G\to\g^*\,,\quad J(g,\zm)=\zm,$$
while in the body coordinates it takes the form $J(g,\zm)=\Ad^*_g\zm$.
It is therefore clear that every $\zm\in\g^*$ is a regular value of the moment map, so that we can use the standard Marsden-Weinstein-Meyer symplectic reduction\cite{Marsden:1974,Meyer:1973} for every $\zm\in\g^*$.

\medskip According to the Marsden-Weinstein-Meyer theorem, on the submanifold
$$N_\zm=J^{-1}(\zm)=G\ti\{\zm\}\simeq G$$
of the cotangent bundle $\sT^*G=G\ti\g^*$, the restriction of $\zw$,
$$\hat\zw=\zw\,\big|_{N_\zm}=i^*_{N_\zm}(\zw),$$
where $i_{N_\zm}:N_\zm\hookrightarrow\sT^*G$ is the canonical injection, is a presymplectic form, invariant with respect to the restricted action of the isotropy subgroup $G_\zm=\{g\in G:\, Ad^*_g(\zm)=\zm\}$, which clearly acts freely and properly on $N_\zm\simeq G$. According to the Marsden-Weinstein-Meyer theorem, the quotient manifold
$$P_\zm=N_\zm/G_\zm\simeq G/G_\zm$$
is a symplectic manifold with the symplectic form $\zw^\zm$ being the reduction of $\hat\zw$.
The canonical submersion
$$p_\zm:N_\zm=G\ti\{\zm\}\to P_\zm\,,\quad p_\zm(g)=[g]=gG_\zm$$
induces the projection $\sT p_\zm:\sT N_\zm\to\sT P_\zm$,
$$\sT p_\zm:\sT N_\zm\ni\big((g,\zm),(v,0)\big)\mapsto ([g],[v])\in\big((G/G_\zm)\ti(\g/\g_\zm)\big)=\sT P_\zm,$$
and the symplectic form $\zw^\zm$ on $P_\zm$ is uniquely determined by the equation $p_\zm^*(\zw^\zm)=\hat\zw$.
Since vectors tangent to $N_\zm$ at $(g,\zm)$ are of the form $(v,0)\in\g\ti\g^*$, the restriction $\hat\zw=\zw_{N_\zm}$ of $\zw$ to the submanifold $N_\zm$ takes in the space coordinates the form (cf. (\ref{csymp}))
$$
\hat\zw_{(g,\zm)}\big((v,0),(v',0)\big)=\bk{\zm}{[v,v']}=-\bk{\ad^*_v(\zm)}{v'}.
$$
Of course, the right-hand-side vanishes for all $v'\in\g$ exactly when $v$ belongs to
$$\g_\zm=\{v\in\g\,|\, \ad^*_v=0\},$$
the Lie algebra of $G_\zm$. Consequently, the reduced symplectic form reads
$$ \zw^\zm_{[g]}([v],[v'])=\bk{\zm}{[v,v']}.$$
There is a canonical diffeomorphism
$$A_\zm:G/G_\zm\to \cO, \quad A_\zm([g])=\Ad^*_g(\zm)$$
of $G/G_\zm$ onto the coadjoint orbit $\cO\subset\g^*$ of $\zm$. It is easy to see that $A_\zm$ is well defined and we will show that this diffeomorphism identifies $\zw^\zm$ with the KKS-symplectic structure $\zw^\cO$ on the orbit. The tangent map $\sT A_\zm:\sT P_\zm\to\sT\cO$ in the space coordinates takes the form
$$\sT A_\zm([g],[v])=\Big(\Ad^*_{g}(\zm),\Ad^*_{g}\big(\ad^*_v(\zm)\big)\Big).$$
But the Lie bracket on $\g$ is $\Ad_g$-invariant, so
$$\Ad^*_{g}\circ\,\ad^*_v=\ad^*_{\Ad_g(v)}\circ\Ad^*_g,$$
and therefore
$$\sT A_\zm([g],[v])=\Big(\Ad^*_g(\zm),\ad^*_{\Ad_g(v)}\big(\Ad^*_g(\zm)\big)\Big).$$
The pull-back of the canonical symplectic form $\zw^\cO$ on the orbit $\cO$ is therefore (cf. (\ref{KKSf}))
\beas &[A_\zm^*(\zw^\cO)]_{[g]}([v],[v'])=\zw^\cO_{\Ad^*_g(\zm)}\Big(\ad^*_{\Ad_g(v)}\big(\Ad^*_g(\zm)\big),
\ad^*_{\Ad_g(v')}\big(\Ad^*_g(\zm)\big)\Big)\\
&=\Bk{\Ad^*_{g}(\zm)}{[\Ad_g(v),\Ad_g(v')]}=\Bk{\Ad^*_{g}(\zm)}{\Ad_g([v,v'])}\\
&=\bk{\zm}{[v,v']}=\zw^\zm_{[g]}([v],[v']).
\eeas
We get therefore the following.
\begin{theorem} The diffeomorphism $A_\zm:P_\zm\to\cO$ is a symplectomorphism of the symplectic manifold $(P_\zm,\zw^\zm)$ onto $(\cO,\zw^\cO)$.
\end{theorem}
\begin{remark}
It is easy to see that the coadjoint orbits of $G$ are canonically symplectomorphic with the coadjoint orbits
of the Lie group $G/H$, where $H$ is a closed Lie subgroup of the center $Z(G)$ of $G$. In particular, coadjoint orbits of $\mathrm{U}(n)$ can be identified with coadjoint orbits of $\mathrm{SU}(n)$, and coadjoint orbits of $$\mathrm{SU}(n)/Z\big(\mathrm{SU}(n))\big)=\mathrm{SU}(n)/\Z_n.$$
\end{remark}
\section{Contactification of coadjoint orbits}
Let us consider again the submanifold $N_\zm=J^{-1}(\zm)$ in $\sT^*G$, which in the right trivialization $\sT^*G\simeq G\ti\g^*$ reads $N_\zm=G\ti\{\zm\}$, so we will often identify $N_\zm$ with $G$. In view of (\ref{Liou}), the restriction $\hat\zvy=\zvy\,\big|_{N_\zm}$ of the Liouville 1-form to $N_\zm$ in the space coordinates looks like
$$\hat\zvy_{g}(v,0)=\bk{\zm}{v},$$
so $\ker(\hat\zvy)_g\simeq\ker(\zm)$ for every $g\in G$.
As we already know, the kernel of $\xd\hat\zvy=\hat\zw$ is $\g_\zm$ at every point, so
\be\label{kernel}\zq(\hat\zvy)=\ker(\hat\zvy)\cap\ker(\xd\hat\zvy)=G\ti\g_\zm^0,\ee
where
$$\g_\zm^0=\{v\in\ker(\zm):\,\ad^*_v(\zm)=0\}=\ker(\zm)\cap\g_\zm.$$
\begin{proposition}\label{p2a}
The subspace $\g_\zm^0$ is a 1-codimensional ideal in the Lie algebra $\g_\zm$, which contains the derived ideal, $[\g_\zm,\g_\zm]\subset\g_\zm^0$.
\end{proposition}
\begin{proof}
We have
$$\bk{\zm}{[\g_\zm,\g_\zm]}=\bk{\ad^*_{\g_\zm}(\zm)}{\g_\zm}=0,$$
so $[\g_\zm,\g_\zm]\subset\g_\zm^0$.
To show that $\g^0_\zm\ne\g_\zm$, consider the submanifold
$\bar N_\zm=G\ti[\zm]_+$ of $\sT^*G\simeq G\ti\g^*$, where $[\zm]_+=\{s\zm:\,s>0\}$.
The restriction $\bar\zw=i^*_{\bar N_\zm}(\zw)$ of the canonical symplectic form $\zw$ on $\sT^*G$ to the submanifold $\bar N_\zm$ looks like (see \ref{csymp})
$$\bar\zw_{(g,s\zm)}\big((v,a\zm),(v',a'\zm)\big)
=a\bk{\zm}{v'}-a'\bk{\zm}{v}+s\bk{\zm}{[v,v']},$$
and it is easy to see that $\ker(\bar\zw_{(g,s\zm)})=\g_\zm^0\ti\{0\}$.
Since the co-rank of the distribution $\ker(\bar\zw)$ in $\sT\bar N_\zm$, i.e.,
$$\dim(\g)+1-\dim(\g_\zm^0)=\dim(P_\zm)+\dim(\g_\zm)+1-\dim(\g_\zm^0)$$
is even and $\dim(P_\zm)$ is even, $\g_\zm^0$ must be of codimension 1 in $\g_\zm$.

\end{proof}
\no Denote with $G_\zm^0$ the connected normal Lie subgroup in $G_\zm$ whose Lie algebra is $\g_\zm^0$.
It follows from (\ref{kernel}) that the characteristic foliation of $\hat\zvy$ consists of orbits of the subgroup
$G^0_\zm$. Under additional assumption that the subgroup $G^0_\zm$ is closed in $G_\zm$ (equivalently, in $G$), the characteristic foliation is simple and, according to Theorem \ref{dfr}, we can do the reduction of the one form $\hat\zvy$ by the characteristic foliation and obtain a nondegenerate 1-form $\zh$ on the reduced manifold
\be\label{M}M=N_\zm/G^0_\zm\simeq G/G^0_\zm.\ee
Since the dimension of $M$ is odd (see Proposition \ref{p2a}), the form $\zh$ is a contact form (cf. Corollary \ref{cc1})) and reads
\be\label{zh}\zh_{[g]}([v])=\bk{\zm}{v}.\ee
Let $\cG=G_\zm/G^0_\zm$ be the quotient Lie group. It is commutative and 1-dimensional. The $G_\zm$-action on $N_\zm$ induces a $\cG$-action on $M=N_\zm/G^0_\zm$. As $\hat\zvy$ is $G_\zm$ invariant, the contact form $\zh$ is $\cG$-invariant and the orbits of $\cG$ consist of the leaves of the kernel of $\xd\zh$, the latter being the reduction of $\xd\hat\zvy=\hat\zw$. In particular, $M/\cG\simeq\cO$. Let $G^e_\zm$ be the connected component of $G_\zm$ containing the identity element, so the connected 1-dimensional Lie group $\cG^e=G^e_\zm/G^0_\zm$ is either $S^1$ or $\R$. If $\mathsf v$ spans the Lie algebra $\mathsf g$ of $\cG^e$ (thus $\cG$), then the 1-parameter subgroup $t\mapsto\exp(t\sv)$ of $\cG$ acts on $M$ as a 1-parameter group of diffeomorphisms with a generator $\hat\sv$ lying in the kernel of $\xd\zh$ and such that $\Ll_{\hat\sv}\zh=0$. It follows that $\hat\sv$ is a multiple of the Reeb vector field $\cR$ of $\zh$, $\hat\sv=c\cR$,  $c\in\R$. We can therefore assume that $\sv$ is chosen in such a way that $\hat\sv=\cR$. As the $\cG^e$-action on $M$ is free and proper, the contact manifold $(M,\zh)$ is principal, so we have a the $\cG^e$-principal bundle $p:M\to M/\cG^e$ which defines a contact-to symplectic reduction of $(M,\zh)$ onto $(M_{red},\zw_{red})$, where $M_{red}=M/\cG^e$ and $p^*(\zw_{red})=\xd\zh$. Note, however, that $M_{red}$ is the coadjoint orbit $\cO$ of $\zm$ if and only if $\cG$ is connected, $\cG=\cG^e$. In general, $(M_{red},\zw_{red})$ is a cover of $(\cO,\zw^\cO)$, being a principal bundle over $\cO$ with the discrete group $\cG/\cG^e$ as the structure group. We can sum up our observations as follows.
\begin{theorem}\label{main0} If the connected normal Lie subgroup $G^0_\zm$ of $G_\zm$ is a closed subgroup in $G_\zm$, then $(M,\zh)$, where $M=(G/G^0_\zm)\ti\{\zm\}$ and $\zh$ is the reduction by $G^0_\zm$ of the restriction $\hat\zvy$ of the Liouville 1-form $\zvy$ on $\sT^*G$ to the submanifold $G\ti\{\zm\}$ (in the right trivialization), is a principal contact manifold with the canonical action of the connected 1-dimensional Lie group $\cG^e=G^e_\zm/G^0_\zm$.
In particular, $(M,\zh)$ is a contactification of the symplectic manifold $(M/\cG^e,\xd\zh/\cG^e)$, which can be canonically identified with  a regular cover of the symplectic coadjoint orbit $(\cO,\zw^\cO)$ of $\zm\in\g^*$ -- a principal bundle with respect to the discrete group $\cG/\cG^e=G_\zm/G^e_\zm$. If $G_\zm$ is connected, then $(M,\zh)$ is a contactification of the symplectic coadjoint orbit $(\cO,\zw^\cO)$.
\end{theorem}
\no Note that, since the contact manifold $(M,\zh)$ turned out to be regular and complete, the above result fully agrees with \cite[Theorem 5.1]{Grabowska:2023}.

\section{The case of a compact Lie group}\label{A}
Let us suppose now that $G$ is compact and connected. It is well known (see, e.g., \cite[Theorem 5.18]{Sepanski:2007}) that in this case the Lie algebra $\g$ is reductive, $\g=\g'\op\mathfrak{z}$, where the derived ideal $\g'=[\g,\g]$ is semisimple and $\mathfrak{z}$ is the center of $\g$. The commutator subgroup $G'=(G,G)$ is a semisimple closed (thus compact) normal subgroup of $G$ with the Lie algebra $\g'$ {\cite[Theorem 5.21]{Sepanski:2007}}, and  $G=G'\ti\T$, where $\T$ is a torus being the identity component of the center $Z(G)$ of $G$ \cite[Theorem 5.22]{Sepanski:2007}.  Moreover, in this case the exponential map $\exp:\g\to G$ is surjective \cite[Theorem 5.12]{Sepanski:2007}, isotropy subgroups $G_\zm$ are connected, and the coadjoint orbits are simply connected \cite{Borel:1954,Filippini:1995} and compact. The canonical KKS symplectic form is in this case a part of a canonical K\"ahler structure \cite{Borel:1954}. Actually, coadjoint orbits are in this case \emph{flag manifolds} understood as certain homogeneous manifolds, as studied already by Ehresmann \cite{Ehresmann:1934}. Note also that there is the canonical identification of adjoint and coadjoint orbits associated with the Killing form, which is nondegenerate in this case. Since the center of $G$ acts trivially on $\g^*$, we can reduce ourselves to the case when $G$ (thus $\g$) is is semisimple and compact (so $Z(G)$ is a finite Abelian subgroup).

\m We will now study the following question: when the Lie subgroup $G_\zm^0$ in $G_\zm$ is closed? Since $G_\zm$ is connected and closed in $G$, thus compact, this reduces to the question whether $G^0_\zm$ is closed in the compact and connected Lie group $G_\zm$. As $G_\zm^0$ is 1-codimensional in $G_\zm$, this is equivalent to the question whether $G_\zm/G_\zm^0\simeq S^1$ is a circle. The covector $\zm$ we can restrict to $\g_\zm$ and view as an element of $\g_\zm^*$.

In general terms, we can start with an arbitrary compact connected Lie group $H$, an element $\zm\in\h^*$ such that $\h^0=\ker(\zm)$ contains $\h'=[\h,\h]$, and ask under what conditions for $\zm$ the quotient group
$H/H_0$, where $H^0=\exp(\h^0)$, is the circle $S^1$ interpreted as the multiplicative group of complex numbers of modulus 1, $z=e^{it}$, $t\in\R$. The Lie algebra of $S^1$ is therefore $\R$ with the trivial Lie bracket and with the exponential map
$$\exp:\R\to S^1\subset\C,\quad \exp(t)=e^{it}.$$
The kernel of this exponential map, $\exp^{-1}(1)$, is the additive subgroup $2\pi\Z$ of $\R$.
Of course, the subgroup $H^0=\exp(V)$ contains the commutator subgroup $H'=(H,H)$ which is a closed normal subgroup (see e.g. \cite[Theorem 5.25]{Sepanski:2007}).
Consider the canonical projection
$$\zF:H\to H/H^0$$
which is a group homomorphism ($\ker(\zm)$ is a Lie ideal in $\h$). Of course, $H^0$ is closed in $H$ if and only if $H/H^0\simeq S^1$. In the latter case the Lie group homomorphism $\zF$ corresponds to a unique Lie algebroid homomorphism
$$\zf:\h\to\R\simeq\h/\h^0=\Lie(S^1)$$
which can be viewed as an element of $\h^*$, so for us $\zf\in\h^*$.
Continuous group morphisms into $S^1$, the latter viewed as the group of complex numbers of modulus 1, we call \emph{characters}, so $\zf\in\h^*$ is a generator of the nontrivial character $\zF$, i.e.,
\be\label{char}\zF(\exp(v))=e^{i\langle \zf,v\rangle}\ee
for any $v\in\h$. In particular, $\zf$ takes values in $2\pi\Z$ on the kernel $\exp^{-1}(e)$ of the exponential map $\exp:\h\to H$,
where $e$ is the unit of $H$. Since $\ker(\zf)$ contains $V$, we have $\ker{\zf}=\ker(\zm)$, so there is $\hbar>0$ such that $\zm=\hbar\zf$.

Conversely, if there exists $\hbar>0$ such that $\zm/\hbar=\zf$ takes values in $2\pi\Z$ on $\exp^{-1}(e)$ and $\ker(\zf)=V$ contains $\h'$, then the map
$$\h\ni v\mapsto e^{i\langle \zf,v\rangle}$$
takes values in $S^1=\U(1)\subset\C$ and (\ref{char}) defines properly a character $\zF$ on $H$.
The condition $(\zm/\hbar)(v)\in 2\pi\Z$ is equivalent to $\zm(v)\in 2\pi\hbar\,\Z$. The additive subgroup $2\pi\hbar\,\Z$ of $\R$ is sometimes denoted $\Z_\hbar$ and we will use this notation in the sequel.
For $H=G_\zm$ we get the following.
\begin{theorem}\label{main}
Let $G$ be a compact and connected Lie group, and let $\zm\in\g^*$. The normal subgroup $G_\zm^0$ in $G_\zm$ corresponding to the Lie ideal $\g_\zm^0=\g_\zm\cap\ker(\zm)$ in $\g_\zm$ is closed if and only if there exists $\hbar> 0$ such that $\zm$ takes $\Z_\hbar$-values at points of $\exp^{-1}(e)$.
\end{theorem}
\begin{remark} There is another approach to the problem of closeness of $H^0$. Suppose that $H/H^0=S^1$, As $S^1$ is commutative, $\ker(\zF)$ clearly contains $H'$, so it goes down to a group homomorphism $\tilde\zF:H/H'\to H/H^0=S^1$. But $H/H'$ is a torus $\T$, which can be identified with the identity component of $Z(H)$, and whose Lie algebra $\t$
is the center $\mathfrak{z}$ of $\h$, so our question reduces to the question whether the image $[H^0]=H^0/H'$ of $H^0$ under the canonical projection $p:H\to \T=H/H'$ is a Lie subgroup of the torus $\T$. Of course, the Lie algebra $[\h^0]=\h^0/\h'$ of the Lie subgroup $[H^0]$ in $\T$ is commutative and can be identified with the center $\mathfrak{z}$ of $\h$. In this way we reduced our problem to the case of tori: when a linear subspace $V$ of the Lie algebra $\t$ of a torus $\T$
corresponds to a closed Lie subgroup of $\T$? In other words, when $\exp(V)$ is closed in $\T$?

It is well known that for tori the kernel $\zG=\ker(\exp)=\exp^{-1}(e)$ of the exponential map $\exp:\t\to\T$ is a free and discrete subgroup of the additive group $(\t,+)$ which spans $\t$. Note that the exponential map is in this case a group homomorphism with respect to the additive group structure on $\t$. Consequently, $\T$ can be identified with $\t/\zG$ also as a group.

\m Conversely, any discrete subgroup $\zG$ of the additive group $(\t,+)$ of a real $n$-dimensional vector space $\t$ is free, with linearly independent free generators. If $\zG$ spans $\t$, we will say that $\zG$ is \emph{total} in $\t$. In this case, the quotient group $\t/\zG$ is an $n$-dimensional torus $\T^n$ whose (commutative) Lie algebra can be identified with $\t$. Lie subalgebras of $\t$ are just linear subspaces $V\subset\t$.
\begin{proposition}\label{l1} Let us consider a torus $\T=\t/\zG$ and let $V$ be a linear subspace of $\t$. The following are equivalent:
\begin{description}
\item{(a)} The image $\exp(V)\subset \T$ of $V$ under the exponential map $\exp:\t\to\t/\zG=\T$ is a closed subgroup in $\T$;
\item{(b)} The discrete subgroup $\zG_0=V\cap\zG$ of is free and total in $V$;
\item{(c)} The image $\zr(\zG)$ of $\zG$ under the canonical projection $\zr:\t\to \t/V$ is a discrete and total subgroup of the vector space $\t/V$;
\item{(d)} The subgroup
$$V+\zG=\{v+\zg:\, v\in V,\ \zg\in\zG\}$$
of $(\t,+)$ is closed.
\end{description}
\end{proposition}
\begin{proof}
$(a)\Rightarrow (b)$ \ Since $\exp(V)$ is a closed subgroup in a torus, it is a torus itself. Moreover, $\exp(V)=V/(V\cap\zG)$, so $\zG_0=V\cap\zG$ is free and total in $V$.

\mn $(b)\Rightarrow (c)$ \ Let $V'$ be a vector subspace of $\t$ which is complementary to $V$, so we have the decomposition $\t=V\op V'$. We can now identify $\t/V$ with $V'$ and $\zr$ with the canonical projection $\t\to V'$ associated with the decomposition. As $\zG$ spans $\t$, the subgroup $\zr(\zG)$ is total in $V'$. We will show that $\zr(\zG)$ is discrete. In the other case, there is a sequence of points $\zg_n$ of $\zG$ such that $v'_n=\zr(\zg_n)\ne 0$ but $v'_n=\zr(\zg_n)\rightarrow 0$.
Let $x_1,\dots,x_k$ be free generators of $\zG_0$. They span $V$, so $V=K+\zG_0$, where $K$ is a compact subset of $K$ defined by
$$K=\{a_1x_1+\dots +a_kx_k\in V:\vert a_i\vert\le 1\}.$$
Hence, we can choose $\zg_n$ to be of the form $v_n+v'_n$, where $v_n\in K$. Since $K$ is compact, passing to a subsequence, we can assume that $v_n\rightarrow w\in K$. Hence, $\zg_n\rightarrow w$. But $\zG$ is closed and discrete, so $w\in\zG_0$ and $\zg_n=w$ for almost all $n$; a contradiction.

\mn $(c)\Rightarrow (d)$ \ The canonical projection $\zr:\t\to\t/V$ is a continuous surjection and $\zr(\zG)$ is closed in $\t/V$, so
$$\t\setminus(V+\zG)=\zr^{-1}\big((\t/V)\setminus\zr(\zG)\big)$$
is open (here, $\setminus$ denotes the setminus).

\mn $(d)\Rightarrow (a)$ \ The map $\exp:\t\to \T=\t/\zG$ is open and surjective, so $\exp(\t\setminus(V+\zG))=T\setminus \exp(V)$ is open.

\end{proof}
\end{remark}
\begin{corollary}
Let $\T=\t/\zG$ be a torus and $\zm\in\t^*$. The Lie subgroup $\exp\big(\ker(\zm)\big)$ is closed in $\T$ if and only if there exists $\hbar> 0$ such that $\zm(\zG)\subset\Z_\hbar$.
In particular, $\zf=\zm/\hbar$ induces a nontrivial character (\ref{char}).
\end{corollary}
\begin{proof} According to Proposition \ref{l1} (c), $\exp\big(\ker(\zm)\big)$ is closed in $\T$ if and only if $\zr(\zG)$ is a discrete and total subgroup of $\t/\ker(\zm)$, where $\zr:\t\to\t/\ker(\zm)$ is the canonical projection. Since $\t/\ker(\zm)$ is 1-dimensional, $\zr(\zG)$ has a single free generator, say $\zr(\zg)$, for some $\zg\in\zG$. Denote $\hbar=\zm(\zg)/2\pi$, so $\zm(\zG)\in\Z_\hbar$. Conversely, if $\zm(\zG)\subset\Z_\hbar$, then $\zm(\ker(\zm)+\zG)=\zm(\zG)\subset\Z_\hbar$, so $\ker(\zm)+\zG$ is closed in $\t$, thus $\exp(\ker(\zm))$ is closed in view of Proposition \ref{l1} (d).

\end{proof}
\noindent Now, for compact connected Lie groups $G$, we are able to make use of Theorem \ref{main0}, since
we know when $G_\zm^0$ is closed in $G_\zm$. Combining Theorems \ref{dfr}, \ref{main0}, and \ref{main}, we can provide a canonical and explicit construction of coadjoint orbit contactifications for $G$.

\mn Let $G$ be a compact and connected Lie group, $\g$ be its Lie algebra, and $\exp:\g\to G$ be the corresponding exponential map. Denote with $\zvy$ the Liouville 1-form on $\sT^*G$, and with $\zw=\xd\zvy$ the canonical symplectic form on $\sT^*G$.
The cotangent lift of the left-regular action of $G$ on itself is a canonical Hamiltonian action of $G$ on the symplectic manifold $(\sT^*G,\zw)$, and the corresponding canonical moment map $J:\sT^*G\to\g^*$ is in the right trivialization $\sT^*G\simeq G\ti\g^*$ just the projection on $\g^*$. Let us consider the coadjoint orbit $\cO$ of a certain $\zm\in\g^*$, equipped with the canonical KKS symplectic form $\zw^\cO$.
\begin{theorem}
Denote with $\hat\zvy$ and $\hat\zw=\xd\hat\zvy$ the restrictions of differential forms $\zvy$ and $\zw=\xd\zvy$, respectively, to the submanifold $N_\zm=J^{-1}(\zm)$ 
of $\sT^*G$, 
$$\hat\zvy=\zvy\,\big|_{N_\zm}=i^*_{N_\zm}(\zvy),\quad\hat\zw=\zw\,\big|_{N_\zm}=i^*_{N_\zm}(\zw)$$
where $i_{N_\zm}:N_\zm\hookrightarrow\sT^*G$ is the obvious immersion.

\mn Then the 2-form $\hat\zw$ is simple (see Definition \ref{forms}), and the corresponding regular reduction (see Theorem \ref{dfr}) of $(N_\zm,\hat\zw)$, i.e., actually the Marsden-Weinstein-Meyer symplectic reduction, gives a symplectic manifold $(P_\zm,\zw^\zm)$ which is canonically isomorphic to $(\cO,\zw^\cO)$.

Moreover, if $\zm\in\g^*$ takes $\Z_\hbar$-values on $\exp^{-1}(e)$, for some $\hbar>0$ and $\exp:\g_\zm\to G_\zm$ being the exponential map, then $\hat\zvy$ is also simple, and the corresponding regular reduction (Theorem \ref{dfr}) again) gives a principal contact manifold $(M,\zh)$ with the structure group $S^1$, which is actually a contactification of $(\cO,\zw^\cO)$.
\end{theorem}
\begin{remark}\label{ul}
In practice, to check the integrality condition for $\zm$, it is convenient to consider the torus $\T=G_\zm/G'_\zm$, then choose generators $x_j$ of the discrete subgroup $\zG\subset\t$ being the kernel of the exponential map $\exp:\t\to\T$, and finally check whether there exists $\hbar>0$ such that $\ul{\zm}(x_j)\in\Z_\hbar$ for all $j$. Here,
$\ul{\zm}\in\t^*$ is just $\zm$ if we understand $\t^*$ as the annihilator $\Ann(\g'_\zm)$ of the Lie ideal  $\g'_\zm\subset\g_\zm$,
$$\t^*=\big(\g_\zm/\g'_\zm\big)^*=\Ann(\g'_\zm)\subset\g_\zm^*.$$
\end{remark}
\mn Let us look closer at the contact manifold $(M,\zh)$ for $\zm(K)=\Z_\hbar$, where $\hbar>0$ and $K$ is the kernel of the exponential map $\exp:\g_\zm\to G_\zm$. Note that we have put equality, so $\hbar$ is the positive generator of $\zm(K)$. Such $\zm$ we will call \emph{$\hbar$-integral}. According to (\ref{M}) and (\ref{zh}),
we can identify $M$ with $G/G_\zm^0$, so
$$\sT M\simeq (G/G_\zm^0)\ti(\g/\g_\zm^0)\quad \text{and}\quad\zh_{[g]}([v])=\bk{\zm}{v}.$$
The leaves of the characteristic foliation of $\xd\zh$ are just represented by the orbits of $G_\zm/G_\zm^0\simeq S^1$, so the vectors of the corresponding Reeb vector field are of the form $([g],[v])$, where $g\in G_\zm$, $v\in\g_\zm$, and $\bk{\zm}{v}=1$.
As $\zm(\zG)=\Z_\hbar$, the $S^1$-action on $M$ induced by the Reeb vector field is therefore an $\big(\R/(2\pi\hbar\,\Z)\big)$-action rather than ($\R/\Z$)-action. This distinction, which is not clearly stated in \cite{Boothby:1958} but carefully explained in \cite{Grabowska:2023}, is important, as it leads to different contact forms. Of course, $\R/\Z$ and $\T_\hbar=\R/\Z_\hbar$ are isomorphic Lie groups and they lead to isomorphic contact manifolds, but not to isomorphic contact forms (isomorphic co-oriented contact manifolds). This is because $2\pi\hbar$ represents the minimal period of the flow generated by the Reeb vector field. Of course, if these periods are different, then the contact forms cannot be isomorphic. In this sense, $M$ is a $\T_\hbar$-principal bundle in the terminology of \cite{Bates:1997,Brylinski:1993} and $(M,h)$ is a \emph{$\T_\hbar$-principal contactification} of $(\cO,\zw^\cO)$.

\mn It is well known (see, e.g., \cite{Boothby:1958,Grabowska:2023}) that $\T_\hbar$-principal contact manifolds $(M,\zh)$ correspond to Hermitian line bundles $L$ over its contact-to-symplectic reductions $(N,\zw)$, where $N=M/\ker(\xd\zh)$ and $\zw$ is the reduction of $\xd\zh$. In this identification, $\zh$ can be viewed as a connection $\n$ in the line bundle $L\to N$, and the symplectic form $\zw$ represents its curvature. In other words, any $\T_\hbar$-principal contact manifold $(M,\zh)$ represents a \emph{prequantization} of the symplectic manifold $(N,\zw)$ in terms of the theory of Geometric Quantization (see e.g. \cite{Bates:1997,Brylinski:1993,Kostant:1970}). The Geometric Quantization was built as an alternative to Dirac's quantization of commutation relations, which cannot be carried out completely, as shown in \cite{Gotay:2000}.
The fundamental result of this theory is a Weil theorem stating that there exists a prequantization of a symplectic manifold $(N,\zw)$, i.e., a Hermitian line bundle $L\to N$ equipped with a connection $\n$ whose curvature is represented by $\zw$, if and only if the symplectic form is \emph{integral}. More precisely, $\hbar>0$ is the minimal positive number such that the class $[\zw]$ of $\zw$ in the de Rham cohomology group $\sH^2(N;\R)$ lies in the image of $\sH^2(N;\Z_\hbar)$ under the canonical homomorphism
$$\sH^2(N;\Z_\hbar)\to\sH^2(N;\R).$$

\noindent Of course, in the actual geometric quantization, $\hbar$ is the Planck constant, but in the mathematical theory $\hbar$ can be an arbitrary positive constant, exactly as we have treated it before. We will speak therefore about \emph{$\hbar$-integrality}.
\begin{definition} Let $\hbar>0$. A symplectic form $\zw$ on $N$ is \emph{$\hbar$-integral} if $\hbar$ is the minimal positive number such that integrating $\zw$ over any compact surface $\mathfrak{S}$ (2-dimensional submanifold) we get a number from $\Z_\hbar$,
$$\int_{\mathfrak{S}}\zw\in\Z_\hbar.$$
\end{definition}
\no More about fiber bundles with $S^1$-fibers you can find in \cite{Brylinski:1993,Kobayashi:1956}.
\begin{corollary}
Let $G$ be a compact connected Lie group and $\zm\in\g^*$ such that $\ul{\zm}(\zG)=\Z_\hbar$ (cf. Remark \ref{ul}) for some $\hbar>0$, where $\zG$ is the kernel of the exponential map for the torus $\T=G_\zm/G'_\zm$,
$\exp:\t\to\T$. Then the KKS symplectic form on $\cO$ is $\hbar$-integral.
\end{corollary}
\no The above result is essentially well known \cite{Bates:1997,Kirillov:2004,Kostant:1970} and the coadjoint orbits of Lie groups have been studied in this context by many other authors (see, e.g., \cite{Gotay:2002,Litsgard:2018,Oblak:2017,Marsden:1982}).

\subsection{The case of unitary groups}
In what follows, we will identify the real vector space of Hermitian operators $A=A^\dag$ on $C^n$ with the dual $\u^*$ of the real Lie algebra $\u$ of anti-Hermitian operators, according to the canonical pairing between Hermitian $A\in \u^*$ and anti-Hermitian $T\in \u$ operators, set this time
\be\label{npair}\la A,T\ran={i}\cdot\tr(A T).\ee
The multiplication by $i$ establishes further vector space isomorphisms (\ref{iso}),
$$\u^*\ni A\mapsto -iA\in \u,$$
which identifies the adjoint $\Ad_U(T)=U T U^\dag$ with the coadjoint
$$(\Ad^*)_U(A)=(\Ad_{U^{-1}})^*(A)=U A U^\dag$$
action of the unitary group $\U=\U(n)$.
\mn Let us consider now $\zm\in\u^*$ of the form
$$\zm=\begin{bmatrix}
   \fbox{$\zl_1$}&\hskip-10pt 0 & \cdots & 0\\
 0  &\hskip-10pt\fbox{$\zl_2$}& \cdots & 0\\
   \vdots &  &  \ddots & \vdots\\
  0 & \cdots &\hskip10pt 0  &\fbox{$\zl_k$}
\end{bmatrix}\,,
$$
where $\zl_j\in\R$ are pairwise different reals, and $\fbox{$\zl_j$}$ is the diagonal matrix $\zl_jI_{d_j}$ of size $d_j>0$. Of course, $\sum_jd_j=n$.
Since the coadjoint action of $U\in\U$ maps $\zm$ to $U\zm\, U^\dag$, we have
$$U\in\mathrm{U}_\zm\ \Leftrightarrow\ U\zm=\zm\, U.$$
Commuting with $\zm$, such $U$ must respect the eigenspaces of $\zm$, so $U\in\U_\zm$ must have the same block-diagonal form as $\zm$,

\mn
\be\label{Gmu}U=\begin{bmatrix}
   U_1& \hskip-10pt 0 & \cdots & 0\\
 0  &\hskip-10pt U_2 & \cdots & 0\\
 \vdots  &  &  \ddots & \vdots\\
  0 & \cdots &\hskip10pt 0 &U_k
\end{bmatrix}\,,
\ee

\mn where $U_j\in\U(d_j)$. Consequently, the Lie algebra $\u_\zm$ consists of block-diagonal matrices

\medskip
\be\label{gmu}T=\begin{bmatrix}
   T_1 &\hskip-10pt 0 & \cdots & 0\\
 0  &\hskip-10pt T_2& \cdots & 0\\
  \vdots &  &  \ddots & \vdots\\
  0 & \cdots &\hskip10pt 0  & T_k
\end{bmatrix}\,,
\ee

\mn where $T_j$ are anti-Hermitian.
\begin{remark}
Since the derived (commutator) normal subgroup $\big(\U(n)'\big)=\big(\U(n),\U(n)\big)$ is the group $\SU(n)$ of matrices $U$ from $\U(n)$ satisfying $\det(U)=1$. The subgroup $\SU(n)$ is 1-codimensional in $\U(n)$ and $\U(n)/\SU(n)\simeq \U(1)=S^1$ is just the circle $\U(1)\subset\C$. The canonical projection of $\U(n)$ onto $\U(n)/\mathrm{SU}(n)\simeq S^1$ looks in this identification like
$$\det:\U(n)\to \U(1),\quad U\mapsto \det(U).$$
The Lie algebra $\su(n)$ of $\SU(n)$ consists of elements from $\u$ with trace 0.

\no All this implies that the commutator subgroup $\U'_\zm$ of $\U_\zm$ consists of matrices
(\ref{Gmu}) such that $U_j\in\SU(d_j)$, and $\U_\zm/\U'_\zm\simeq \T^k$ is the $k$-dimensional torus
$\T^k=(S^1)^{\ti k}$. The canonical surjective homomorphism $\zt:\U_\zm\to\T^k=(S^1)^{\ti k}$ reads
$$\zt(U)=\big(\det(U_1),\dots,\det(U_k)\big).$$
Since for any square matrix $T$ we have $\det\big(e^T\big)=e^{\tr(T)}$,
the corresponding homomorphism $D\zt:\u_\zm\to\R^k$ takes the form
\be\label{Z}D\zt(T)=\big(\tr(T_1),\dots,\tr(T_k)\big),\ee
so the Lie algebra $\u'_\zm$ consists of matrices (\ref{gmu}) such that $T_j\in\su(d_j)$.
Vanishing exactly on $\u'_\zm$, the map (\ref{Z}) induces an isomorphism
$$\zr:(\u_\zm/\u'_\zm)\to\R^k,\quad \zr([T])=\big(\tr(T_1),\dots,\tr(T_k)\big).$$
where $\R^k$ is understood as the Lie algebra of the torus $\T^k$.
\end{remark}
\no As the exponential map $\exp:\u_\zm\to\U_\zm$ is clearly
$$\exp(T)=\begin{bmatrix}
   \exp(T_1)& \hskip-10pt 0 & \cdots & 0\\
 0  &\hskip-10pt \exp(T_2) & \cdots & 0\\
 \vdots  &  &  \ddots & \vdots\\
  0 & \cdots &\hskip10pt 0 &\exp(T_k)
\end{bmatrix},
$$
and $\det(e^T)=e^{\tr(T)}$, the composition $(\zt\circ\exp):\u_\zm\to\T^k$ reads
$$(\zt\circ\exp)(T)=\big(e^{\tr(T_1)},\dots,e^{\tr(T_k)}\big).$$
It is clear now, that $T$ is in the kernel of the above map if and only if $\tr(T_j)=-2\pi in_j$ for some $n_1,\dots,n_k\in\Z$. According to (\ref{npair}),
$$\bk{\zm}{T}={i}\sum_j\big(\zl_j\cdot\tr(T_j)\big),
$$
so $\zm$ takes values in $\Z_\hbar$ on the kernel of $\zt\circ\exp$ if and only if
$$\sum_j(\zl_j\cdot n_j)\in \hbar\,\Z,$$
for any $n_1,\dots,n_k\in\Z$. Of course, the latter means exactly that all $\zl_j\in\hbar\,\Z$. We get therefore the following.
\begin{theorem}\label{Th}
Let $A$ be a Hermitian $n\ti n$ matrix, $A\in\u^*(n)$, and $\hbar>0$. Then the canonical KKS symplectic form $\zw^\cO$ (cf. (\ref{KKSf})) on the $\U(n)$-coadjoint orbit $\cO$ through $A$ is $\hbar$-integral if and only if the eigenvalues of $\zm$ generate the subgroup $\hbar\,\Z$ in $\R$. In other words, $A$ has the diagonal form
\be\label{A}A=\hbar\cdot\begin{bmatrix}
   \fbox{$\zl_1$}&\hskip-10pt 0 & \cdots & 0\\
 0  &\hskip-10pt\fbox{$\zl_2$}& \cdots & 0\\
   \vdots &  &  \ddots & \vdots\\
  0 & \cdots &\hskip10pt 0  &\fbox{$\zl_k$}
\end{bmatrix}\,,
\ee

\mn for some $\zl_1,\dots,\zl_k\in\Z$, where $\fbox{$\zl_j$}$ is the diagonal matrix of the size $d_j$ and $\zl_j$ on the diagonal. If $A$ is a quantum state, then additionally  $\zl_j\in\N$
and
$$\frac{1}{\hbar}=\sum_j(\zl_j\cdot d_j).$$
\end{theorem}
\no Let us recall that in the above theorem we identify the space $\u^*(n)$ with the real vector space of Hermitian operators on $\C^n$ \emph{via} (\ref{npair}). For the formulation of the following corollary we use Definition \ref{forms}, Theorem \ref{dfr}, and the right trivializations (\ref{triv}),(\ref{triv*}).
\begin{corollary}
Let $A\in\u^*(n)$ be as in (\ref{A}), satisfying the conditions described in Theorem \ref{Th}, and let $N$ be the submanifold in $\sT^*\big(\U(n)\big)$, which in the right trivialization reeds
$$N=\U(n)\ti\{A\}\subset\U(n)\ti\u^*(n)\simeq\sT^*\big(\U(n)\big).$$
Then $N$ is a simple submanifold of $\big(\sT^*\big(\U(n)\big),\zvy\big)$, where $\zvy$ is the Liouville 1-form, and the regular reduction $\zh=\zvy^N_{red}$ of the restriction $\hat\zvy$ of $\zvy$ to $N$ is a contact form
on the reduced manifold $M=N^\zvy_{red}$. This is actually the reduction by the natural action of the compact connected Lie subgroup $\U^{\,0}_A(n)=\big(\U(n)\big)_A^0$ of $\U(n)$ consisting of unitary matrices of the form $\exp(T)$, where $T$ is anti-Hermitian of the form (\ref{gmu}) with
$$\sum_j\big(\zl_j\cdot\tr(T_j)\big)=0,$$
so $M\simeq \U(n)/\U_A^0$.

Moreover, the contact manifold $(M,\zh)$ is $\T_\hbar$-principal, i.e., the flow induced by the Reeb vector field is periodic with the minimal period $2\pi\hbar$ along each orbit and its contact-to-symplectic reduction is canonically isomorphic to the symplectic manifold $(\cO,\zw^\cO)$, the coadjoint orbit of $A\in\u^*(n)$.
\end{corollary}
\no The above corollary can be applied to constructing canonical contactifications of coadjoint orbits of quantum states. It is clear that $A\in\us(n)$ satisfies the conditions of Theorem \ref{Th} for some $\hbar>0$ if and only if the subgroup
\be\label{cond} \exp_A=\big\{e^{tiA}\,\big|\, t\in\R\big\}\ee
of $\U(n)$ is compact (thus is a circle).
\begin{corollary}
If $\zr$ is a quantum state (density matrix) in $\C^n$ such that $\exp_\zr$ is compact, then the coadjoint orbit of $\U(n)$ through $\zr$ admits a canonical $\U(1)$-principal (thus compact) contactification.
\end{corollary}
\begin{example} Consider the group $\U(3)$ and $\zr\in\u^*(3)$,
$$\zr=\frac{1}{6}\begin{bmatrix}
   1& 0 &  0\\
 0  & 2 & 0\\
  0  & 0  &3
\end{bmatrix}\,.
$$
Clearly, $A=\zr$ satisfies condition (\ref{cond}). The group $\U_\zr(3)$ consists of matrices
\be\label{u3}U=\begin{bmatrix}
   e^{i\zt_1}& 0 &  0\\
 0  & e^{i\zt_2} & 0\\
  0  & 0  & e^{i\zt_3}
\end{bmatrix}\,,
\ee
thus it is a maximal torus $K$ in $\U(3)$, and $\U^{\,0}_\zr(3)$ is a 1-codimensional
closed subgroup $K^0$ in this torus consisting of matrices (\ref{u3}) for which
\be\label{eq1}\zt_1+2\zt_2+3\zt_3=0.\ee
Hence, the coadjoint orbit $\cO$ through $\zr$ is $\U(3)/K$ and its compactification $M$ is $\U(3)/K^0$.
Of course, we can reduce ourselves to $\SU(3)$ and $\hat K=K\cap\SU(3)$ which consists of matrices (\ref{u3})
for which
\be\label{eq2}
\zt_1+\zt_2+\zt_3=0.\ee
Consequently, $\hat K^0=K^0\cap\SU(3)$ consists of matrices (\ref{u3}) satisfying both equations (\ref{eq1}) and (\ref{eq2}). We have a 1-parameter solution
$$\zt_1=a,\ \zt_2=-a,\ \zt_3=a,$$
So $\hat K^0$ is $S^1$. Hence, $M=\SU(3)/S^1$ and
$$\cO=\SU(3)/(S^1\ti S^1)=M/S^1.$$
Generally, all generic coadjoint orbits of $\U(n)$ (equivalently, $\SU(n)$) are diffeomorphic to $\U(n)/K$,
where $K$ is a maximal torus in $\U(n)$ (equivalently, $\SU(n)/\hat K$), and are fiber bundles
over the complex projective space $\C\P^{n-1}$ with fibers being generic orbits of $\U(n-1)$ (see the nice survey
\cite{Bernatska:2008}). Of course, these orbits are generally not isomorphic as symplectic manifolds. In our case $\cO$ is a (nontrivial) fiber bundle over $\C\P^{2}$ with the fiber $\C\P^1$.
\end{example}

\section{Lagrangians for compact Hamiltonian systems}\label{B}
There is considerable interest in studying Hamiltonian systems on symplectic manifolds $(N,\zw)$ which are more general than the standard phase spaces $N=\sT^*Q$, or even Poisson manifolds. Hamiltonian vector fields are defined for arbitrary Hamiltonian $H:N\to M$ without problems. What causes problems is the Lagrangian picture for such systems, since a globally defined action functional is difficult or impossible to find, even in the case when the physically identified configuration space $Q$ can be defined for $N$. The reason why one can do no better than get local Lagrangian descriptions is that one is unable to find well-defined canonical coordinates of the usual kind
on the entire phase space. This in turn is bound to happen when, in the language of differential geometry, the symplectic structure is given by a closed but non-exact two-form. It has been found \cite{Balachandran:1992,Zaccaria:1983} that in this situation, if the symplectic form obeys certain prequantization conditions, then classically a global Lagrangian formulation can be achieved on a suitable enlarged configuration space. This approach applies also to quantum mechanics.
In the context of our work, these are contactifications which can serve as enlarged phase spaces for some symplectic manifolds, even in the compact case, where symplectic forms are never exact.
\begin{example}[Hopf fibration]\label{Hf}
In the context of a magnetic monopole, one considers $(\R^3)^\ti=\R^3\setminus\{0\}$ with the closed two-form
$$B=x\cdot\xd y\we\xd z + y\cdot\xd z\we\xd x + z\cdot\xd x\we\xd y,$$
representing the magnetic field of a magnetic monopole centered at the origin. This is the standard volume form $\vol_{S^2}$ on the unit sphere $S^2\subset(\R^3)^\ti$, so
$$\int_{S^2}B=4\pi,$$
thus $B$ is not exact. In the spherical coordinates,
$$x=\sin(\zy)\cos(\zf),\quad y=\sin(\zy)\sin(\zf),\quad z=\cos(\zy),$$
where $\zy\in[0,\zp]$, $\zf\in[0,2\pi]$, we have
$$\zw_{S^2}=\sin(\zy)\xd\zy\we\xd\zf.$$
The KKS symplectic form $\zw^{S^2}$ on $S^2=\C\P^1$ regarded as a coadjoint orbit of $\U(2)$ a contactification $S^3\subset\C^2$, with the contact form $\zvy_0$ being the restriction of the Liouville 1-form
$$\zvy=\half\sum_{k=1}^2\big(q^k\xd p_k-p_k\xd q^k\big)$$
on $\C^2$ to $S^3$, where $z_k=q^k+i\cdot p_k$, $k=1,2$. Moreover, $S^3$ is canonically an $\U(1)$-principal bundle over $S^2$ with the projection $\zt:S^3\to S^3/\U(1)=S^2$. In fact, $\zt:S^3\to S^2$ is an example of a nontrivial principal bundle with the typical fiber $S^1$, called the \emph{Hopf fibration}. The pullback $\zt^*(\zw^{S^2})$ is exact and $\zt^*(\zw^{S^2})=\xd\zvy_0$.

\mn To describe the projection $\zp:S^3\to S^2$ explicitly, consider the map
$$p:\C^2\to\C\ti\R, \quad p(\za,\zn)=\big(\za\bar\zn,|\za|^2-|\zn|^2\big),
$$
where $\C\ni\zn\mapsto\bar\zn\in\C$ is the complex conjugation. Interpreting $\C$ as $\R^2$, we can think that $p:\R^4\to\R^3$. Note that if $|\za|^2+|\zn|^2=1$, then
$$|\za\bar\zn|^2+\big(|\za|^2-|\zn|^2\big)^2=1,$$
so $p$ maps the unit sphere $S^3$ in $C^2=\R^4$ into the unit sphere $S^2$ in $\C\ti\R=\R^3$. Moreover, $\zt=p\,\big|_{S^3}:S^3\to S^2$ is surjective and maps $(\za,\zn)\in S^3$ and $(\za',\zn')\in S^3$ to the same point if and only if $(\za',\zn')=(\zl\za,\zl\zn)$ for some $\zl\in\C$, $|\zl|=1$. Consequently, fibers of this projections are circles $S^1\simeq\U(1)=\{ z\in\C\,:\, |z|=1\}$, being the orbits of the $\U(1)$-action. One can also interpret $S^3\subset\C^2$, consisting of $(\za,\zn)\in\C^2$ such that $|\za|^2+|\zn|^2=1$, as the group $\SU(2)$ \emph{via} the identification
$$S^3\ni(\za,\zn)\rightarrowtail\begin{bmatrix}
\za & -\bar\zn\\
\zn & \bar\za
\end{bmatrix}\in\SU(2)\,.
$$
In this identification, $S^2$ is the quotient $\SU(2)/\U(1)$, where $\U(1)$ is the subgroup of diagonal elements,
$$\begin{bmatrix}
\za & 0\\
0 & \bar\za
\end{bmatrix}\,,
$$
$|\za|^2=1$. One can also use the direct parametrization of the three-sphere in $\C^2$ by the \emph{Euler angles},
\beas && z_1=e^{-\frac{i}{2}\big(\zc+\zf\big)}\cos(\zy/2),\\
&& z_2=e^{-\frac{i}{2}\big(\zc-\zf\big)}\sin(\zy/2)\,,
\eeas
where $(\zvy,\zf)$ are spherical coordinates on $S^2$ and $\zc\in[0,4\pi)$. Hence, the two form $\zw_0=\xd\zvy_0$ on $S^3$, being the restriction of
$$\zw=-\frac{i}{2}\Big(\xd \bar z_1\we\xd z_1+\xd\bar z_2\we\xd z_2\Big)$$
to $S^3$, projects onto
\beas && -\frac{i}{2}\Big(\xd\big(e^{\frac{i}{2}\zf}\cos(\zy/2)\big)\we\xd\big(e^{-\frac{i}{2}\zf}\cos(\zy/2)\big)\Big)\\
&&-i\Big(\xd\big(e^{-\frac{i}{2}\zf}\sin(\zy/2)\big)\we\xd\big(e^{\frac{i}{2}\zf}\sin(\zy/2)\big)\Big)\\
&& = \frac{1}{2}\sin(\zy/2)\cos(\zy/2)\xd\zy\we\xd\zf=\frac{1}{4}\vol_{S^2}.
\eeas
\end{example}

\mn Let us go now to the Hamiltonian mechanics on symplectic manifolds admitting a contactification, like the two dimensional sphere above. Suppose that $(M,\zh)$ is a contactification of a symplectic manifold $(N,\zw)$ and $\zt:M\to N=M/\cF$ is the corresponding surjective submersion, where $\cF$ is the simple foliation by trajectories of the Reeb vector field $\cR$. Let $H:N\to\R$ be a Hamiltonian function and $X_H$ the corresponding Hamiltonian vector field, $i_{X_H}\zw=-\xd H$. For paths $\zg:[t_0,t_1]\to M$ we define the action functional
\be\label{actionf}\zF[\zg]=\int_{t_0}^{t_1}\hat H\big(\zg(t)\big)\xd t-\int_\zg\zh=\int_{t_0}^{t_1} \Big(\hat H\big(\zg(t)\big)-\Bk{\zh\big(\zg(t)\big)}{\dot\zg(t)}\Big)\xd t\,,
\ee
where $\hat H=H\circ\zt$ is the pullback of $H$. For a variation $\zg+\ze\xd\zg$ such that $\xd\zg(t_0)=\xd\zg(t_1)=0$, we have
\bea&&\frac{\xd}{\xd \ze}\,\Big|_{\ze=0}\zF[\zg+\ze\zd\zg]\label{part0}\\
&&=\int_{t_0}^{t_1} \Bigg(\frac{\xd}{\xd \ze}\,\Big|_{\ze=0}\Big(\hat H\big(\zg(t)+\ze\zd\zg(t)\big)-\Bk{\zh\big(\zg(t)+\ze\zd\zg(t)\big)}{\dot\zg(t)
+\ze(\zd\zg)^{\centerdot}(t)}\Big)\Bigg)\xd t\nn\\
&&=\int_{t_0}^{t_1}\Big(\Bk{\xd\hat H\big(\zg(t)\big)}{\zd\zg(t)}+
\xd\zh\big(\zg(t)\big)\big(\dot\zg(t),\zd\zg(t)\big)\Big)\xd t\,.\nn
\eea
To see the latter equality, let us write in local coordinates $(y^i)$ on $M$,
$$\frac{\xd}{\xd\ze}\,\Big|_{\ze=0}\Big(\Bk{\zh\big(\zg+\ze\zd\zg\big)}
{\dot\zg+\ze(\zd\zg)^{\centerdot}}\Big)
=\frac{\pa\zh_i}{\pa y^j}\big(\zg\big)(\zd\zg)^j\dot\zg^i
+\zh_i\big(\zg\big)\frac{\xd(\zd\zg)^i}{\xd t}.
$$
But
$$\int_{t_0}^{t_1}\Big(\zh_i\big(\zg\big)\frac{\xd(\zd\zg)^i}{\xd t}\Big)\xd t=
\Big(\zh_i\big(\zg\big)(\zd\zg)^i\Big)\,\Big|_{t_0}^{t_1}-
\int_{t_0}^{t_1}\Big(\frac{\pa\zh_i}{\pa y^j}\big(\zg\big)\dot\zg^j(\zd\zg)^i\Big).
$$
The first summand on the right side is zero, so
\beas&&\int_{t_0}^{t_1} \Bigg(\frac{\xd}{\xd\ze}\,\Big|_{\ze=0}\Big(\Bk{\zh\big(\zg+\ze\zd\zg\big)}
{\dot\zg+\ze(\zd\zg)^{\centerdot}}\Big)\Bigg)\xd t
=\int_{t_0}^{t_1}\Bigg(\Big(\frac{\pa\zh_i}{\pa y^j}-\frac{\pa\zh_i}{\pa y^j}\Big)\big(\zg\big)(\zd\zg)^j\dot\zg^i\Bigg)\xd t\\
&&=\int_{t_0}^{t_1}\Big(\xd\zh\big(\zg(t)\big)\big(\zd\zg(t),\dot\zg(t)\big)\Big)\xd t.
\eeas
Since $\zd\zg$ in (\ref{part0}) is arbitrary, we get from (\ref{part0}) the Euler-Lagrange equation
\be\label{ELe}
i_{\dot\zg}\,\xd\zh+\xd\hat H=0.
\ee
The above equation does not determine a solution uniquely, even if an initial condition is fixed. This is because $\xd\zh$ has a kernel, whose characteristic foliation consist of trajectories of the Reeb vector field. However, the equation projected on $N$ is exactly the Hamilton equation with the Hamiltonian $H$. In other words, projection of solutions of (\ref{ELe}) are trajectories of the Hamiltonian vector field $X_H$. To sum up, we formulate the following.
\begin{theorem}
Let $(M,\zh)$ be a contactification of a symplectic manifold $(N,\zw)$ and $\zt:M\to N$ be the corresponding fibration. For every Hamiltonian $H:N\to\R$,  denote with $\hat H$ its pullback, $\hat H=H\circ\zt$. Then every stationary point $\zg:(t_0,t_1)\to M$ of the action functional (\ref{actionf}) projects \emph{via} $\zt$ onto a trajectory of the Hamiltonian vector field $X_H$ on $N$.
\end{theorem}
\begin{example}
Let $H(x,y,z)=(z+1)/4$ be a Hamiltonian on $S^2$. In the notation of Example \ref{Hf} and Theorem \ref{cpn}, the pullback $\hat H=H\circ \zt$ on $S^3$ reads
$$\hat H(z_1,z_2)=\big(|z_1|^2-|z_2|^2+1\big)/4=|z_1|^2/2=\frac{(q^1)^2+(p_1)^2}{2}.$$
The stationary points $\zg:[t_0,t_1]\to S^3$ of the action functional
\beas &&\zF[\zg]=\int_{t_0}^{t_1}\hat H\big(\zg(t)\big)\xd t-\int_\zg\zh
=\int_{t_0}^{t_1}\frac{\big(q_1(t)\big)^2+\big(p_1(t)\big)^2}{2}\xd t\\
&&-\frac{1}{2}\int_{t_0}^{t_1}\Big(q^1(t)\dot p_1(t)-\dot p_1(t)\dot q^1(t)+
q^2(t)\dot p_2(t)-\dot p_2(t)\dot q^2(t)\Big)\xd t
\eeas
satisfy the equation (\ref{ELe}), i.e., there is a function $a:[t_0,t_1]\to\R$ such that
\beas && \Bk{\dot q^1(t)\pa_{q^1}+\dot p_1(t)\pa_{p_1}+\dot q^2(t)\pa_{q^2}+\dot p_2(t)\pa_{p_2}}
{\big(\xd q^1\we\xd p_1+\xd q^2\we\xd p_2\big)_0}\\
&&+q^1(t)\xd q^1+p_1(t)\xd p_1+\frac{a(t)}{2}\xd \big((q^1)^2+(p_1)^2+(q^2)^2+(p_2)^2\big)=0\,,
\eeas
where $\big(\xd q^1\we\xd p_1+\xd q^2\we\xd p_2\big)_0$ is the restriction of the 2-form $\zw=\xd q^1\we\xd p_1+\xd q^2\we\xd p_2$ on $\C^2$ to $S^3$.
It follows that there is a function $a:[t_0,t_1]\to\R$
$$\dot q^2(t)=-a(t)p_2,\quad \dot p_2(t)=a(t)q^2,\quad \dot p_1(t)=\big(1+a(t)\big)q^1(t),\quad \dot q^1(t)=-\big(1+a(t)\big)p_1(t).$$
Hence,
$$\zg(t)=\Big(e^{-i\big(t+A(t)\big)}z_1,e^{iA(t)}z_2\Big),$$
where $\dot A(t)=a(t)$. The projected curve $\ul\zg(t)=\zt\big(\zg(t)\big)$ in $\R^3$ reads
$$\ul\zg(t)=\Big(\cos(t)x_0+\sin(t)y_0,\cos(t)y_0-\sin(t)x_0,z_0\Big).$$
These are exactly the trajectories of the Hamiltonian vector field $X_H$ on $S^2$ with respect to the symplectic form $\frac{1}{4}\vol_{S^2}$.  These are the trajectories of the Hamiltonian vector field $X_{H_0}=y\pa_x-x\pa_y$ for $H_0=4H=z+1$ with respect to $\vol_{S^2}=\zw^{S^2}/4$. Indeed,
$$i_{X_{H_0}}B=i_{y\pa_x-x\pa_y}\big(x\cdot\xd y\we\xd z + y\cdot\xd z\we\xd x + z\cdot\xd x\we\xd y\big)=
-(x^2+y^2+z^2)\xd z+z(x\xd x+y\xd y+z\xd z).$$
But $x\xd x+y\xd y+z\xd z$ vanishes on $S^2$, so
$$i_{X_{H_0}}\vol_{S^2}=-\xd z.$$
\end{example}
\section{Conclusions and outlook}
We have studied some questions related to the concept of a contactification of a symplectic manifold, together with a geometric construction of contactifications of coadjoint orbits of Lie groups. This construction is based on methods of the Marsden-Weinstein-Meyer symplectic reduction. However, there are obstacles to carrying out such a construction, having a clear topological interpretation. These obstructions are equivalent to the celebrated Dirac quantization conditions in the case of coadjoint orbits of compact groups.

We have also shown that contactifications provide a nice geometrical tool for a Lagrangian description of Hamiltonian systems, even if the symplectic form is not exact (does not possess a `vector potential'). A nice example related to the magnetic monopole, and geometrically to the Hopf fibration, is provided.

\mn We should admit, however, that one fundamental problem remains open, namely the question of existence of a (connected) contactification of compact symplectic manifolds. The point is that contactifications can be \emph{a priori} awkward, weird, and topologically very complicated, as every open submanifold of a contact manifold is contact. Thus, our challenge for the future work is to prove the following conjecture.
\begin{conjecture}
A compact symplectic manifold $(N,\zw)$ admits a connected contactification if and only if it satisfies the Dirac quantization condition
$$\big[\zw/2\pi\hbar\big]\in\sH^2(N,\Z)$$
for some $\hbar>0$.
\end{conjecture}
\no Of course, applications of the described `lagrangianization' to particular, physically interesting Hamiltonian systems is another important task for the forthcoming work.

\vskip.5cm
\noindent Katarzyna Grabowska\\\emph{Faculty of Physics,
University of Warsaw,}\\
{\small ul. Pasteura 5, 02-093 Warszawa, Poland} \\{\tt konieczn@fuw.edu.pl}\\
https://orcid.org/0000-0003-2805-1849\\

\noindent Janusz Grabowski\\\emph{Institute of Mathematics, Polish Academy of Sciences}\\{\small ul. \'Sniadeckich 8, 00-656 Warszawa,
Poland}\\{\tt jagrab@impan.pl}\\  https://orcid.org/0000-0001-8715-2370
\\

\noindent Marek Ku\'s\\
\emph{Center for Theoretical Physics, Polish Academy of Sciences,} \\
{\small Aleja Lotnik{\'o}w 32/46, 02-668 Warszawa,
Poland} \\{\tt marek.kus@cft.edu.pl}
\\

\noindent Giuseppe Marmo\\
\emph{Dipartimento di Fisica ``Ettore Pancini'', Universit\`{a} ``Federico II'' di Napoli} \\
\emph{and Istituto Nazionale di Fisica Nucleare, Sezione di Napoli,} \\
{\small Complesso Universitario di Monte Sant Angelo,} \\
{\small Via Cintia, I-80126 Napoli, Italy} \\
{\tt marmo@na.infn.it}\\
https://orcid.org/0000-0003-2662-2193
\\

\end{document}